\RequirePackage{plautopatch}
\documentclass{amsart}
\usepackage{mathtools}
\usepackage{amsmath}
\usepackage{amsthm}
\usepackage{autobreak}
\usepackage{cleveref}

\theoremstyle{plain}
\newtheorem{theorem}{Theorem}[section]
\newtheorem{proposition}[theorem]{Proposition}
\newtheorem{corollary}[theorem]{Corollary}
\newtheorem{lemma}[theorem]{Lemma}
\newtheorem*{theorem*}{Theorem}

\theoremstyle{definition}
\newtheorem{definition}[theorem]{Definition}

\theoremstyle{remark}
\newtheorem{remark}[theorem]{Remark}
\newtheorem*{remark*}{Remark}
\newtheorem*{acknowledgment}{Acknowledgment}

\newcommand{\R}{\mathbb{R}}
\newcommand{\N}{\mathbb{N}}
\newcommand{\D}{\mathcal{D}}
\newcommand{\X}{\mathcal{X}}
\newcommand{\F}{\mathcal{F}}
\newcommand{\K}{\mathcal{K}}
\renewcommand{\L}{\mathcal{L}}
\DeclareMathOperator{\lip1}{Lip_1}
\DeclareMathOperator{\pr}{pr}
\DeclareMathOperator{\id}{id}
\newcommand{\kf}{d_{\operatorname{KF}}}
\DeclareMathOperator{\pd}{PartDiam}
\DeclareMathOperator{\od}{ObsDiam}
\newcommand{\dpr}{d_{\operatorname{P}}}
\DeclareMathOperator{\diam}{diam}
\DeclareMathOperator{\supp}{supp}
\DeclareMathOperator{\conc}{conc}
\DeclareMathOperator{\dis}{dis}
\DeclareMathOperator{\bdd}{Bdd}
\newcommand{\dinf}[1]{{d^{#1}_\infty}}
\newcommand{\dconc}{d_{\operatorname{conc}}}
\newcommand{\hausp}[1]{{\left(#1\right)_{\operatorname{H}}}}
\newcommand{\hauspdinf}[1]{{\hausp{\dinf{#1}}}}
\newcommand{\mmid}{\mathrel{}\middle|\mathrel{}}
\renewcommand{\epsilon}{\varepsilon}
\renewcommand{\phi}{\varphi}

\title{Geometry of geometric data set I}
\author{Shigeaki Yokota}
\date{\today}
\thanks{The author was supported by JSPS KAKENHI Grant Number 22H04942}

\begin{document}
\begin{abstract}
    Hanika, Schneider, and Stumme introduced geometric data set as a generalization of metric measure space for the computation of the observable diameter, and extended the observable distance between metric measure spaces to that between geometric data sets. In this paper, we begin by proving the non-separability of the observable distance between geometric data sets. We then extend the box distance between mm-spaces to that between geometric data sets and prove its completeness and non-separability.
\end{abstract}
\maketitle

\section{Introduction}
Gromov \cite{gromov2007met} developed the geometry of mm-spaces, based on the concentration of measure phenomenon and theory of collapsing manifold. A triplet $(X, d_X, \mu_X)$, or simply $X$, is an \emph{mm-space} if $d_X$ is a complete separable metric on $X$, and $\mu_X$ is a Borel probability measure on the metric space $(X, d_X)$. He defined the observable distance $\dconc$ based on measure concentration and the box distance $\Box$ based on collapsing theory on the set of all mm-spaces, say $\X$, and constructed a natural compactification of $(\X, \dconc)$. Various properties of these distances and spaces are known, with particular attention given to the separability of $(\X, \dconc)$ and the completeness and separability of $(\X, \Box)$.

Pestov \cite{pestov2008gds} treated data as mm-spaces and used the observable diameter (see Definition~\ref{def:obsdiam}) to explain a form of the curse of dimensionality through the concentration of measure phenomenon. However, efficient computation of the observable diameter of an mm-space $X$ is not yet known, and straightforward computation over the set of all 1-Lipschitz continuous functions, say $\lip1(X)$, in the definition takes exponential time with respect to the number $\#X$ of set of data sample $X$. He proposed restricting the function family to compute the observable diameter.

Following this, Hanika, Schneider, and Stumme \cite{hanika2022gds} defined the geometric data set as a generalization of mm-spaces to compute the curse of dimensionality through the observable diameter as follows: A triplet $(X, F_X, \mu_X)$, or simply $X$, is defined to be a \emph{geometric data set} if $F_X$ is a set of real valued functions on $X$ such that
\begin{equation*}
    d_X(x, y) \coloneqq \sup_{f \in F_X} |f(x) - f(y)|,\ x,y\in X,
\end{equation*}
is a complete separable metric, and if $\mu_X$ is a Borel probability measure on the metric space $(X, d_X)$. An mm-space can be interpreted as a special case where $F_X = \lip1(X, d_X)$. They extended the observable distance between mm-spaces to geometric data sets and similarly generalized the observable diameter to geometric data sets. They then confirmed the characteristics of the observable diameter of geometric data sets as a statistical measure.

Hanika et al.'s paper aims to utilize the observable diameter as a statistical measure and does not generally prove the properties of the observable distance. While the separability of $(\X, \dconc)$ is known, the topological properties of $(\D, \dconc)$ are entirely unknown. Similarly, although the completeness and separability of $(\X, \Box)$ are established, the box distance on $\D$ has yet to be defined. \emph{In this paper, we prove the non-separability of $(\D, \dconc)$ and further extend the box distance to $\D$, showing its completeness and non-separability.}

We state a theorem about the non-separability of $(\D, \dconc)$. For a set $A \subset \R$, define a \emph{single-point geometric data set} $\ast_A$ as
\begin{equation*}
    \ast_A \coloneqq \{\ast\}, \quad F_{\ast_A} = \{\ast \mapsto x \mid x \in A\}, \quad \mu_{\ast_A} = \delta_{\ast}.
\end{equation*}
Then, we have the following.

\begin{theorem}\label{thm:NonSeparability}
    The set $\{\ast_A \mid A \subset \N\}$ is an uncountable discrete subset of $(\D, \dconc)$. In particular, $(\D, \dconc)$ is non-separable.
\end{theorem}

This is in contrast to the separability of $(\X, \dconc)$.

Next, we extend the box distance $\Box$ to that between geometric data sets. The box distance between mm-spaces $X$ and $Y$ is defined using only distance and measure. Nakajima showed the following representation \cite{nakajima2022coupling}:
\begin{equation*}
    \Box(X, Y) = \inf_{\pi, S} \max\{1 - \pi(S), \dis S\},
\end{equation*}
where $\pi$ runs over all couplings (Definition~\ref{def:Coupling}) between $\mu_X$ and $\mu_Y$, and $S$ runs over all closed sets in $X \times Y$. The \emph{distortion} `$\dis S$' of $S$ is defined as
\begin{equation*}
    \dis S \coloneqq \sup\{|d_X(x, x') - d_Y(y, y')|; (x, y), (x', y') \in S\}.
\end{equation*}

To extend the box distance $\Box$ to $\D$, the value $\dis S$ with respect to the distances $d_X$ and $d_Y$ needs to be replaced by the values with respect to the function families $F_X$ and $F_Y$. Introducing an extended pseudo-distance for real-valued functions $f$ and $g$ on $S$,
\begin{equation*}
    \dinf{S}(f, g) \coloneqq \sup_{x \in S} |f(x) - g(x)|,
\end{equation*}
and denoting its Hausdorff distance by $\hauspdinf{S}$, we show the following lemma:
\begin{lemma}\label{lemma:2Lip}
    We have
    \begin{equation*}
        \dis S = 2\hauspdinf{S}(\lip1(X) \circ \pr_1, \lip1(Y) \circ \pr_2).
    \end{equation*}
\end{lemma}

Here, $\pr_n$ denotes the projection from the product to the $n$-th component for $n=1,2$, and $F \circ f$ represents the set of all compositions of each element in the function family $F$ with $f$. Using this lemma, we extend the box distance to $\D$ as follows:
\begin{definition}
    We define the box distance between two geometric data sets $X$ and $Y$ as
    \begin{equation*}
        \Box(X, Y) \coloneqq \inf_{\pi, S} \max\{1 - \pi(S), 2\hauspdinf{S}(F_X \circ \pr_1, F_Y \circ \pr_2)\},
    \end{equation*}
    where $\pi$ runs over all couplings between $\mu_X$ and $\mu_Y$, and $S$ runs over all closed sets in $X \times Y$.
\end{definition}

The box distance defined here satisfies the axioms of a distance. Lemma~\ref{lemma:2Lip} implies that the box distance defined here on $\D$ is a generalization of the conventional box distance. It is known that the inequality $\dconc \leq \Box$ holds between the conventional observable distance and the box distance, and this inequality is easily proved also for geometric data sets.

Furthermore, the box distance on $\D$ satisfies the following property.
\begin{theorem}
    $(\D, \Box)$ is complete but not separable.
\end{theorem}
The non-separability of $(\D, \Box)$ is shown by $\dconc \leq \Box$ and Theorem~\ref{thm:NonSeparability}. In particular, this contrasts with the completeness and separability of $(\X, \Box)$.

This paper is organized as follows. In Chapter 2, we introduce basic definitions, such as mm-space and coupling. In Chapter 3, we define geometric data set, describe basic geometric data sets such as the single-point geometric data set, and discuss the (semi)continuity of the observable diameter of geometric data sets. Next, we refer to the definition of the partial order relation between geometric data sets called feature order, and construct the quotient geometric data set, which is the dominated geometric data set, from the subset of features. The quotient geometric data set is used to prove the completeness of the box distance. In Chapter 4, we confirm the non-separability of the observable distance between geometric data sets and apply similar techniques to Nakajima's characterization \cite{nakajima2022coupling} of the observable distance between mm-spaces for that between geometric data sets. In Chapter 5, we extend the box distance to geometric data sets and show that this extension retains completeness but not separability.

In the next paper \cite{gds2}, we will consider how much the techniques on $\X$ can be applied to the observable distance and box distance on $\D$. Concerning $\X$, the natural compactification of $(\X, \dconc)$ was constructed by Gromov \cite{gromov2007met}, and Shioya \cite{shioya2016mmg} showed that this compactification is metrizable. The obvious obstacle in applying the same method to geometric data sets is the non-separability of $(\D, \dconc)$ and $(\D, \Box)$. Here, by considering a suitable subset of $\D$, it is possible to keep $\dconc$ and $\Box$ separable. This subset is the set of all elements $X$ of $\D$ with a function family $F_X$ that is appropriately ``closed" with respect to a suitable function family $\L \subset \lip1(\R)$, so let us denote it by $\D_\L$. It can be shown that $(\D_\L, \dconc)$ and $(\D_\L, \Box)$ are separable. Furthermore, by using the same techniques as the compactification of $(\X, \dconc)$, a natural compactification of $(\D_\L, \dconc)$ can be constructed and shown to be metrizable.

\begin{acknowledgment}
    The author would like to thank Professor Takashi Shioya for many helpful suggestion and guidance.
\end{acknowledgment}

\section{Metric Measure Space and Coupling}
We denote by $\pr_n$ for $n=1,2,\ldots$ the projection to the $n$-th coordinate. We also set \begin{align*}
    \pr_{n_1,n_2,\ldots,n_m} & \coloneqq (\pr_{n_1},\pr_{n_2}\ldots,\pr_{n_m}) \\
    \pr_{\leq m}             & \coloneqq \pr_{1,2,\ldots,m}
\end{align*} for $m=1,2,\ldots$ and $n_1,n_2,\ldots,n_m=1,2,\ldots$~.

Let be $X$ a topological space. We denote by $\F(X)$ the set of closed subsets of $X$ and by $\K(X)$ the set of compact subsets of $X$.

\begin{definition}[Push-forward] Let $(X,\mu)$ be a measure space and $Y$ a measurable space. For a measurable map $f\colon X\to Y$, we define \begin{equation*}
        f_*\mu(A) \coloneqq \mu(f^{-1}(A)) \ \textrm{for any measurable subset}\ A\subset Y.
    \end{equation*} We call $f_*\mu$ \emph{the push-forward of $\mu$ by $f$}.
\end{definition}

\begin{definition}[Support] Let $X$ be a topological space and $\mu$ a measure on $X$. We define the \emph{support $\supp \mu$ of $\mu$} by \begin{equation*}
        \supp \mu \coloneqq \bigl\{x\in X\mid \mu(U) > 0 \ \textrm{for all open neighbourhood $U$ of $x$}\bigr\}.
    \end{equation*}
\end{definition}

\begin{lemma}[{\cite[Proposition 2.3]{nakajima2022coupling}}]\label{thm:suppImage} Let $X$ and $Y$ be two topological spaces and let $f\colon X\to Y$ be a continuous map. If a Borel measure $\mu$ on $X$ satisfies \begin{equation*}
        \mu(X\setminus\supp \mu) = 0,
    \end{equation*} then we have \begin{equation*}
        \supp f_*\mu = \overline{f(\supp \mu)}.
    \end{equation*}
\end{lemma}

\begin{definition}[Metric measure space, mm-space]
    A triple $(X,d,\mu)$ is called a \emph{metric measure space}, or \emph{mm-space} for short, if the pair $(X,d)$ is a complete separable metric space, $\mu$ is a Borel probability measure, and $\supp \mu = X$.
\end{definition}
\begin{remark}
    We sometimes say that $X$ is an mm-space, for which the associated distance function is denoted by $d_X$ and the measure by $\mu_X$.
\end{remark}

\begin{definition}[Prohorov distance] Let $\mu$ and $\nu$ be two Borel probability measures on a separable distance space $X$. The \emph{Prohorov distance} between $\mu$ and $\nu$ is defined by \begin{equation*}
        \dpr(\mu,\nu) \coloneqq \inf\bigl\{\varepsilon \geq 0\mid \mu\bigl(U_X(A;\varepsilon)\bigr)\geq \nu(A)-\varepsilon \ \textrm{for any Borel subset}\ A\subset X\bigr\}.
    \end{equation*}
\end{definition}

\begin{definition}[Ky Fan metric] Let $(X,\mu)$ be a measure space and $Y$ a metric space. We define the \emph{Ky Fan metric} $\kf^\mu$ on the set of $\mu$-measurable maps from $X$ to $Y$ by\begin{equation*}
        \kf^\mu(f,g) \coloneqq \inf\left\{\epsilon \geq 0\mid \mu(\left\{x\in X\mid d_Y(f(x),g(x)) > \epsilon\right\}) \leq \epsilon\right\}
    \end{equation*} for any two $\mu$-measurable maps $f,g\colon X\to Y$.
\end{definition}

\begin{lemma}[{\cite[Lemma 1.26]{shioya2016mmg}}]\label{lem:ProhorovLeqKyFan}
    Let $(X,\mu)$ be a Borel probability measure space, $Y$ a metric space, and $f,g\colon X\to Y$ maps. Then, we have \begin{equation*}
        \dpr(f_*\mu,g_*\mu) \leq \kf^\mu(f,g).
    \end{equation*}
\end{lemma}

\begin{definition}[Coupling]\label{def:Coupling}
    Let $(X,\mu),(Y,\nu)$ be Borel probability measure spaces. A Borel probability measure $\pi$ on $X\times Y$ is called a \emph{coupling} of $\mu$ and $\nu$ if $(\pr_1)_*\pi = \mu$ and $(\pr_2)_*\pi = \nu$. We denote by $\Pi(\mu,\nu)$ the set of couplings of $\mu$ and $\nu$.
\end{definition}

\begin{lemma}[Prohorov's theorem]
    Let $X$ be a separable metric space and $M$ a subset of the set of the Borel probability measures on $X$. Then, the following (1) and (2) are equivalent to each other:\begin{enumerate}
        \item For any $\epsilon > 0$, there exists a compact set $K$ in $X$ such that $\mu(X\setminus K) < \epsilon$ for any $\mu\in M$.
        \item $M$ is relatively compact with respect to $\dpr$.
    \end{enumerate}
\end{lemma}

\begin{lemma}\label{lem:CouplingCompact}
    Let $(X,\mu),(Y,\nu)$ be Borel probability measure spaces. Then, $\Pi(\mu,\nu)$ is $\dpr$-compact.
\end{lemma}

\begin{definition}[Weak Hausdorff convergence] Let $X$ be a metric space and $S_n$, $S$ be closed sets on $X$, $n = 1,2,\ldots$~. We say that $S_n$ converges to $S$ in the \emph{weak Hausdorff} sense as $n\to\infty$ if the following (1) and (2) are satisfied.
    \begin{enumerate}
        \item For any $x\in S$, we have \begin{equation*}
                  \lim_{n\to\infty} d_X(x,S_n) = 0.
              \end{equation*}
        \item For any $x\in X\setminus S$, we have \begin{equation*}
                  \liminf_{n\to\infty} d_X(x,S_n) > 0.
              \end{equation*}
    \end{enumerate}
\end{definition}

\begin{theorem}[{\cite[Theorem 5.2.12]{beer1993topologies}}]\label{thm:weakHausdorffSubSequence} Any sequence of closed sets on a complete separable metric space has a subsequence that converges in the weak Hausdorff sense.
\end{theorem}

\begin{proposition}[{\cite[Lemma 3.9]{nakajima2022coupling}}]\label{thm:weakHausdorffMeasureBound}
    Let $X$ be a metric space, $\mu$ be a Borel probability measure on $X$, and $S_n,S$ be closed sets on $X$, $n = 1,2,\ldots$~. If $S_n$ converges to $S$ in the weak Hausdorff sense, then we have\begin{equation*}
        \mu(S) \geq \limsup_{n\to\infty} \mu(S_n).
    \end{equation*}
\end{proposition}

\section{Geometric Data Set}
\subsection{Feature Space and Geometric Data Set}

\begin{definition}[Tame]
    Let $X$ be a non-empty set and $F$ a non-empty subset of the power set $\R^X$. We define a function $d_F\colon X\times X \to [0,+\infty]$ by \begin{equation*}
        d_F(x,y) \coloneqq \sup_{f\in F} |f(x)-f(y)|,\quad x,y\in X.
    \end{equation*}
    $d_F$ is symmetric and satisfies a triangle inequality.

    $F$ is said to be \emph{tame} if $d_F(x,y) < +\infty$ for any $x,y\in X$.
\end{definition}

\begin{definition}[Feature space]
    A pair $(X,F)$ is called a \emph{feature space} if $X$ is a non-empty set, $F\subset \R^X$ is tame, and $d_F$ is a distance function. For a feature space $(X,F)$, any element of $F$ is called a \emph{feature}.
\end{definition}
\begin{remark}
    We sometimes say that $X$ is a feature space, for which the associated feature set is denoted by $F_X$ and the distance function $d_{F_X}$ by $d_X$.
\end{remark}

\begin{proposition}\label{thm:lip1=d}
    Let $X$ be a metric space. Then we have $d_{\lip1(X)} = d_X$.
\end{proposition}
\begin{proof}
    Take any $x,y\in X$. For any $f\in\lip1(X)$, it follows that
    \begin{equation*}
        |f(x) - f(y)| \leq d_X(x,y),
    \end{equation*} which proves $d_{\lip1(X)}(x,y) \leq d_X(x,y)$. We define $g\in\lip1(X)$ as \begin{equation*}
        g(z) \coloneqq d_X(x,z),\ z\in X.
    \end{equation*}
    Since \begin{equation*}
        d_X(x,y) = |g(x)-g(y)| \leq d_{\lip1(X)}(x,y),
    \end{equation*} we have $d_{\lip1(X)}(x,y) = d_X(x,y)$. The arbitrariness of $x$ and $y$ completes the proof.
\end{proof}

\remark{
    For a feature space $(X,F)$, we have $F\subset \lip1(X,d_F)$.
}

Let $(X,F)$ be a feature space, $p\in\lip1(\R)$, $Y$ a non-empty set, and $\varphi\colon Y\to X$ a map. We denote $p\circ F \coloneqq \{p\circ f \mid f\in F\}$ and $F\circ\varphi\coloneqq\{f\circ \varphi\mid f\in F\}$.

\begin{definition}[Product feature space]
    The \emph{product feature space} of $(X,F)$ and $(Y,G)$ is defined by $(X\times Y,F\otimes G)$, where \begin{equation*}
        F\otimes G \coloneqq F\circ\pr_1 \cup\ G\circ\pr_2.
    \end{equation*}
\end{definition}
\begin{remark}
    $d_{F\otimes G}$ coincides with $d_\infty$ on $X\times Y$, where
    \begin{equation*}
        d_\infty\left((x_1,y_1),(x_2,y_2)\right) \coloneqq \max\{d_F(x_1,x_2),d_G(y_1,y_2)\}
    \end{equation*}
    for all $x_1,x_2\in X$ and $y_1,y_2\in Y$.
\end{remark}

\begin{definition}[Geometric data set]
    We call a triple $(X,F,\mu)$ a \emph{geometric data set} if $(X,d_F,\mu)$ is an mm-space.
\end{definition}
\begin{remark}
    We sometimes say that $X$ is a geometric data set, for which the associated feature set is denoted by $F_X$, the measure on $X$ by $\mu_X$, and the distance function $d_{F_X}$ by $d_X$.
\end{remark}

\begin{definition}[Singleton geometric data set]
    For the singleton $\ast$, its element is also represented by $\ast$. For a subset $A\subset \R$, we define a \emph{singleton geometric data set} by \begin{equation*}
        \ast_A \coloneqq (\ast, \{\ast \mapsto a \mid a\in A\}, \delta_\ast),
    \end{equation*} where $\delta_\ast$ is the Dirac measure at $\ast$.
\end{definition}

\begin{definition}[Product geometric data set]
    Let $X,Y$ be two geometric data sets. The \emph{product geometric data set} $X\times Y$ of $X$ and $Y$ is defined by \begin{equation*}
        F_{X\times Y} \coloneqq F_X\otimes F_Y,\quad \mu_{X\times Y} \coloneqq \mu_X\otimes \mu_Y.
    \end{equation*}
\end{definition}

\begin{proposition}\label{prop:pointwiseEqKf}
    Let $X$ be an mm-space and $Y$ be a metric space. On $\lip1(X,Y)$, the pointwise convergence and $\kf^X$-convergence are equivalent to each other. In particular, $\kf^X$ is a metrization of not only the convergence of measure but also the pointwise convergence.
\end{proposition}
\begin{proof}
    First, we show that if a sequence in $\lip1(X,Y)$ converges pointwise, then it $\kf^X$-converges. Let $f_n, f\in \lip1(X,Y)$, $n=1,2,\ldots$, and assume that $f_n$ converges pointwise to $f$ as $n\to\infty$. Take any $\varepsilon > 0$. There exists a compact subset $K\subset X$ such that $\mu_X(K) > 1-\varepsilon$. By the compactness of $K$, there exist finitely many points
    \begin{equation*}
        x_1,x_2,\ldots,x_M\in K \text{ such that } K\subset\bigcup_{m=1}^M U_X(x_m; \varepsilon).
    \end{equation*}
    Since $f_n$ converges pointwise to $f$ as $n\to\infty$ and $M$ is finite, there exists $N\in\N$ such that \begin{equation*}
        d_Y(f_n(x_m),f(x_m)) < \varepsilon\quad \textrm{for any}\ m=1,2,\ldots,M \text{ and } n\geq N.
    \end{equation*}
    For any $x\in K$, there exists $m\in\{1,2,\ldots,M\}$ with $d_X(x,x_m) < \varepsilon$. It follows that $d_Y(f_n(x),f(x)) < 3\varepsilon$ for $n\geq N$. Since $\mu(X\setminus K) < \varepsilon$ and $x\in K$ is arbitrary, we have $\kf^X(f_n,f) \leq 3\varepsilon$ for any $n\geq N$.

    Next, we show that if a function sequence in $\lip1(X,Y)$ $\kf^X$-converges, then it converges pointwise. Let $f_n, f\in \lip1(X,Y)$, $n=1,2,\ldots$, such that $f_n$ does not converge pointwise to $f$ as $n\to\infty$. There exists $x\in X$, $\delta > 0$, and a subsequence $\{f_{\iota(n)}\}_n$ such that \begin{equation*}
        d_Y(f_{\iota(n)}(x),f(x)) > \delta\ \textrm{for}\ n=1,2,\ldots.
    \end{equation*}
    Take $r>0$ so small that $\mu(U_X(x,\delta - r)) > r$. For any $\varepsilon \in (0,r)$ and any $n\in\N$, \begin{align*}
        \mu_X(\{x\in X\mid d_Y(f_{\iota(n)}(x),f(x)) > \varepsilon\}) & \geq \mu_X(U_X(x;\delta-\varepsilon)) \\
                                                                      & \geq r > \varepsilon.
    \end{align*} Since it shows that $\kf^X(f_{\iota(n)},f) \geq r$, $n=1,2,\ldots$, we prove that $f_n$ does not $\kf^X$-converge to $f$ as $n\to\infty$.
\end{proof}

\begin{remark}
    Hereafter, we will denote the closure of $F$ with respect to the pointwise convergence or the topology of convergence in measure by $\overline{F}$ for $F\subset\lip1(X)$.
\end{remark}

\begin{definition}[Isomorphism of geometric data set]
    Two geometric data sets $X$ and $Y$ are said to be \emph{isomorphic}, and we write $X\simeq Y$, if there exists a Borel measurable map $\varphi\colon X\to Y$ such that $\varphi_*\mu_X = \mu_Y$ and $\overline{F_Y}\circ\varphi = \overline{F_X}$. Such $\varphi$ is called an \emph{isomorphism}. Denote by $\D$ the set of isomorphism classes of geometric data sets.
\end{definition}

\begin{definition}[Feature order, dominate]
    We say that a geometric data set $X$ \emph{dominates} a geometric data set $Y$, and we write $Y\preceq X$, if there exists a Borel measurable map $f\colon X\to Y$ such that $f_*\mu_X = \mu_Y$ and $F_Y\circ f \subset \overline{F_X}$. Such an $f$ is called a \emph{domination}. The relation $\preceq$ is called the \emph{feature order relation}.
\end{definition}

\begin{proposition}[{\cite[Proposition 3.5]{hanika2022gds}}]
    The relation $\preceq$ is a partial order on $\D$.
\end{proposition}

\begin{remark}
    From Proposition~\ref{thm:lip1=d}, any mm-space $X$ can be regarded as a geometric data set $(X,\lip1(X),\mu_X)$. We define \begin{align*}
        D(X)  & \coloneqq (X,\lip1(X),\mu_X)\in\D,\ \mathrel{\textrm{for}} X\in\X, \\
        D(\X) & \coloneqq \{D(X)\mid X\in\X\},                                     \\
        \X(X) & \coloneqq (X,d_X,\mu_X)\in\X,\ \mathrel{\textrm{for}} X\in\D.
    \end{align*}
\end{remark}

\begin{proposition}\label{thm:domintationIs1Lip}
    Let $X$ and $Y$ be geometric data sets. Then a domination $p\colon X\to Y$ is 1-Lipschitz continuous.
\end{proposition}
\begin{proof}
    Take any $x,y\in X$ and any $\varepsilon > 0$. There exists $g\in F_Y$ such that \begin{equation*}
        d_Y(p(x),p(y)) < |g(p(x)) - g(p(y))| + \varepsilon.
    \end{equation*} Since $g\circ p\in \overline{F_X}$, there exists $f\in F_X$ such that \begin{equation*}
        |g\circ p(x) - f(x)| < \varepsilon,\quad |g\circ p(y) - f(y)| < \varepsilon.
    \end{equation*} Therefore, \begin{align*}
        d_Y(p(x),p(y)) & < |g\circ p(x) - g\circ p(y)| + \varepsilon                  \\
                       & < |f(x) - f(y)| + 3\varepsilon \leq d_X(x,y) + 3\varepsilon.
    \end{align*}

    By the arbitrariness of $\varepsilon$, we have $d_Y(p(x),p(y)) \leq d_X(x,y)$. Also, by the arbitrariness of $x$ and $y$, we prove $p\in\lip1(X,Y)$.
\end{proof}

\begin{remark}
    In \cite[Definition 2.10]{shioya2016mmg}, we say that an mm-space $X$ dominates an mm-space $Y$ if there exists $f\in\lip1(X,Y)$ such that $f_*\mu_X = \mu_Y$. In this case, $f$ is a domination from $(X,\lip1(X),\mu_X)$ to $(Y,\lip1(Y),\mu_Y)$. Let $g$ is a domination from a geometric data set $X'$ to a geometric data set $Y'$. Since Proposition~\ref{thm:lip1=d}, $g$ is domination from the mm-space $(X',d_{X'},\mu_{X'})$ to $(Y',d_{Y'},\mu_{Y'})$.

    However, in general, a domination over mm-spaces is not necessarily a domination over geometric data sets. For example, $\id_\ast$ is a domination from the mm-space $(\ast,d_\ast,\delta_\ast)$ to $(\ast,d_\ast,\delta_\ast)$ but not a domination from $\ast_{\{0\}}$ to $\ast_{\{1\}}$.
\end{remark}

\subsection{Observable Diameter}
\begin{definition}[Observable diameter] \label{def:obsdiam} Let $\mu$ be a Borel probability measure and $\kappa\in[0,1]$ a real number. The \emph{$(1-\kappa)$-partial diameter of $\mu$} is defined by\begin{equation*}
        \pd(\mu;1-\kappa) \coloneqq \inf\left\{\diam I\mid I\subset \R\colon \textrm{Borel measurable},\ \mu(I)\geq 1-\kappa\right\}.
    \end{equation*}
    For a geometric data set $X$, the \emph{$\kappa$-observable diameter of $X$} is defined as \begin{equation*}
        \od(X;-\kappa) \coloneqq \sup\left\{\pd(f_*\mu_X;1-\kappa)\mid f\in F\right\}
    \end{equation*}
\end{definition}

\begin{remark}
    The observable diameter of an mm-space $X$ is defined by \begin{equation*}
        \od(X;-\kappa) \coloneqq \sup\left\{\pd(f_*\mu_X;1-\kappa)\mid f\in \lip1(X)\right\}
    \end{equation*} from \cite[Definition 2.13]{shioya2016mmg}. We see that $\od(X;-\kappa) = \od(D(X);-\kappa)$.
\end{remark}

\begin{definition}[Lévy family] A sequence $\{X_n\}_{n=1}^\infty$ of geometric data sets is said to be a \emph{Lévy family} if \begin{equation*}
        \lim_{n\to\infty} \od(X_n;-\kappa) = 0
    \end{equation*} for any real number $\kappa\in (0,1)$.
\end{definition}

\begin{lemma}[{\cite[Lemma 3.1(1)]{ozawa2015limit}}]\label{lem:partDiamParamEneq} Let $\mu$ and $\nu$ be Borel probability measures on $\R$. For any real numbers $\kappa\in(0,1)$ and $\delta > \dpr(\mu,\nu)$, we have\begin{equation*}
        \pd(\mu;1-(\kappa+\delta)) \leq \pd(\nu;1-\kappa) + 2\delta.
    \end{equation*}
\end{lemma}

\begin{lemma}[{\cite[Lemma 3.1(2)]{ozawa2015limit}}]\label{thm:partDiamLeftConti} Let $\mu$ be a Borel probability measure on $\R$. The $\alpha$-partial diameter $\pd(\mu;\alpha)$ of $\mu$ is left-continuous with respect to a real number $\alpha\in(0,1)$.
\end{lemma}
\begin{proposition} Let $X$ be a geometric data set. The $\kappa$-observable diameter of $\mu$ is right-continuous with respect to a real number $\kappa\in(0,1)$.
\end{proposition}
\begin{proof}
    By Lemma~\ref{thm:partDiamLeftConti}, we have\begin{align*}
        \od(X;-\kappa) & = \sup_{f\in F_X} \pd(X;1-\kappa)                               \\
                       & = \sup_{f\in F_X} \lim_{n\to\infty} \pd(X;(1-\kappa)-1/n)       \\
                       & \leq \limsup_{n\to\infty} \sup_{f\in F_X} \pd(X;1-(\kappa+1/n)) \\
                       & =\limsup_{n\to\infty} \od(X;-(\kappa+1/n)).
    \end{align*} From the monotonic non-increasing property of $\od(\mu;-\kappa)$ with respect to $\kappa$, the $\kappa$-observable diameter $\od(\mu;-\kappa)$ of $\mu$ is right-continuous with respect to $\kappa$.
\end{proof}

\subsection{Quotient Geometric Data Set}

\begin{lemma}\label{thm:quotientDistance}
    Let $X$ be a complete separable metric space and $d'$ a pseudometric on $X$. If $d'(x,y) \leq d_X(x,y)$ for any $x,y\in X$, then there exist a complete separable metric space $Y$ and a 1-Lipschitz continuous map $f\colon X\to Y$ such that the following \textup{(1)}, \textup{(2)}, and \textup{(3)} hold.\begin{enumerate}
        \item The image $f(X)$ is tight on $Y$.
        \item For any $x,y\in X$, we have $d_Y(f(x),f(y)) = d'(x,y)$.
        \item Let $Z$ be a metric space and $g\colon X\to Z$ a 1-Lipschitz continuous map. If $d_Z(g(x),g(y)) \leq d'(x,y)$ for any $x,y\in X$, then there exists a unique 1-Lipschitz continuous map $\tilde{g}\colon Y\to Z$ such that $\tilde{g}\circ f = g$. In particular, $Y$ is unique up to isometries.
    \end{enumerate}
\end{lemma}
\begin{proof}
    Let us define $Y$ and $f$. We define an equivalence relation $\mathord{\sim}$ on $X$ by saying that $x\sim y$ if and only if $d'(x,y)=0$. Let $p\colon X\to X/\mathord{\sim}$ be the natural projection and
    \begin{equation*}
        \tilde{d}(p(x),p(y)) \coloneqq d'(x,y) \text{ for all } x,y\in X.
    \end{equation*}
    Then, $(X/\mathord{\sim},\tilde{d})$ is a well-defined metric space, and $p$ is 1-Lipschitz continuous. Let $Y$ be the completion of $(X/\mathord{\sim},\tilde{d})$, $q\colon X/\mathord{\sim}\to Y$ the natural distance preserving map, and $f \coloneqq q\circ p$. We see that $f$ is 1-Lipschitz continuous.

    Since $f(X) = q\circ p(X) = q(X/\mathord{\sim})$, (1) is obvious. In particular, $Y$ is separable. Since $q$ is a distance preserving map, (2) is also clear. We prove (3). Let Z be a metric space and $g\colon X\to Z$ a 1-Lipschitz continuous map. We assume that
    \begin{equation*}
        d_Z(g(x),g(y)) \leq d'(x,y) \text{ for any } x,y\in X.
    \end{equation*}
    We define
    \begin{equation*}
        g'\colon X/\mathord{\sim} \to Z \text{ by } g'(p(x)) \coloneqq g(x) \text{ for } x\in X.
    \end{equation*}
    Since
    \begin{equation*}
        d_Z(g(x),g(y)) \leq d'(x,y) = 0 \text{ for any } x,y\in X \text{ with } x\sim y,
    \end{equation*}
    the well-definedness of $g'$ holds. From the definition by $g'$, we have the continuity of $g'$. By the universal property of the completion $Y$, there exists a unique continuous map $\tilde{g}\colon Y\to Z$ satisfying $\tilde{g}\circ q = g'$. Since $g'\circ p = g$, we have $\tilde{g}\circ f= g$. Let us show the 1-Lipschitz continuity of $\tilde{g}$. For all $x,y\in X$, we have \begin{align*}
        d_Z(\tilde{g}(f(x)),\tilde{g}(f(y))) & = d_Z(g(x),g(y)) \leq d'(x,y) = d_Y(f(x),f(y)).
    \end{align*} By (1), this completes the proof of the 1-Lipschitz continuity of $\tilde{g}$. Finally, we prove the uniqueness of $\tilde{g}$. Take any maps $h,k\colon Y\to Z$ that $h\circ f = k\circ f$, that is, $h\circ q\circ p = k\circ q\circ p$. By the projectivity of $p$, we have $h\circ q = k\circ q$. From the universal property of $Y$, we see $k=h$.
\end{proof}

\begin{definition}[Quotient geometric data set, quotient domination]
    Let $X$ be a geometric data set and $G\subset\overline{F_X}$ a subfamily. A geometric data set $Y$ is a \emph{quotient geometric data set} of $X$ by $G$ and a map $f\colon X\to Y$ is a \emph{quotient domination} if the following conditions \textup{(1)} and \textup{(2)} are satisfied:
    \begin{enumerate}
        \item The equality $F_Y\circ f = G$ holds.
        \item For any geometric data set $Z$ and any domination $g\colon X\to Z$, if $F_Z\circ g\subset\overline{G}$, then there exists a unique domination $\tilde{g}\colon Y\to Z$ such that $\tilde{g}\circ f = g$.
    \end{enumerate}
    By definition, if a quotient geometric data set of $X$ by $G$ exists, it is unique up to isomorphism, so we denote it by $X/G$.

    Moreover, for a geometric data set $X$ and $G \subset Lip_1(X)$, even if $G \not\subset F_X$, we also define $X/G$ as $(X,Lip_1(X),\mu_X)/G$. In this case, we still refer to it as a quotient geometric data set; however, the quotient domination is not defined.
\end{definition}
\begin{proposition}\label{thm:QuotientSpaceExist}
    Let $X$ be a geometric data set. For any subset $G\subset F_X$, there exist a quotient geometric data set $Y$ of $X$ by $G$ and a quotient domination $f\colon X\to Y$.
\end{proposition}
\begin{proof}
    By Lemma~\ref{thm:quotientDistance}, there exist a complete separable metric space $Y$ and a 1-Lipschitz continuous map $f\colon X\to Y$ such that the following (a), (b), and (c) hold. \begin{enumerate}
        \renewcommand{\labelenumi}{(\alph{enumi})}
        \item The image $f(X)$ is tight on $Y$.
        \item For any $x,y\in X$, we have $d_Y(f(x),f(y)) = d_G(x,y)$.
        \item Let $Z'$ be a metric space and $g\colon X\to Z'$ a 1-Lipschitz continuous map. If $d_{Z'}(g(x),g(y)) \leq d_G(x,y)$ for any $x,y\in X$, then there exists a unique 1-Lipschitz continuous map $\tilde{g}\colon Y\to Z'$ such that $\tilde{g}\circ f = g$.
    \end{enumerate}
    We set
    \begin{equation*}
        F_Y \coloneqq \{h\in \lip1(Y)\mid h\circ f\in G\} \text{ and } \mu_Y \coloneqq f_*\mu_X.
    \end{equation*}
    The conditinon (c) for $Z'=\R$ implies (1). In particular, we see that $F_Y\circ f \subset \overline{F_X}$. By Lemma~\ref{thm:suppImage} and (a), we have $\supp \mu_Y = Y$. Let us prove $d_Y=d_{F_Y}$. Take any $x,y\in X$. We have \begin{align*}
        d_Y(f(x),f(y)) & = d_G(x,y) = \sup_{g\in G} |g(x)-g(y)|                            \\
                       & = \sup_{h\in F_Y} |h\circ f(x)-h\circ f(y)| = d_{F_Y}(f(x),f(y)).
    \end{align*} By (a), we obtain $d_Y=d_{F_Y}$. In particular, $Y$ is a geometric data set, and $f$ is a domination.

    We prove (2). Let $Z$ be a geometric data set and $g\colon X\to Z$ a domination such that $F_Z\circ g\subset \overline{G}$. By applying (c) for $Z'=Z$, there exists a unique 1-Lipschitz continuous map $\tilde{g}\colon Y\to Z$ such that $\tilde{g}\circ f = g$. Since \begin{equation*}
        F_Z\circ \tilde{g}\circ f = F_Z\circ g \subset \overline{G} = \overline{F_Y\circ f} = \overline{F_Y}\circ f
    \end{equation*} and given the universal property of $f$, it follows that $F_Z\circ \tilde{g} \subset \overline{F_Y}$. Since
    \begin{equation*}
        \mu_Z = g_*\mu_X = \tilde{g}_*f_*\mu_X = \tilde{g}_*\mu_Y,
    \end{equation*}
    we find that $\tilde{g}$ is a domination. The uniqueness of $\tilde{g}$ is obvious since Proposition~\ref{thm:domintationIs1Lip} and (c).
\end{proof}

\section{Observable Distance}
\subsection{Definition and Characterization of Observable Distance}

\begin{definition}[Parameter]
    We set $I\coloneqq [0,1]$ and denote by $\lambda$ the one-dimentional Lebesgue measure on $I$. Let $(X,\mu)$ be a Borel probability measure space. A Borel measurable map $\phi\colon I\to X$ is called a \emph{parameter} of $\mu$ if $\phi_*\lambda = \mu$.
\end{definition}

\begin{definition}[Observable distance]
    Let $X$ and $Y$ be two geometric data sets. We define the \emph{observable distance} $\dconc(X,Y)$ between $X$ and $Y$ as the infimum of
    $\hausp{\kf^\lambda}(F_X\circ\phi,F_Y\circ\psi)$,
    where $\phi$ and $\psi$ run over all parameters of $\mu_X$ and $\mu_Y$, respectively.
\end{definition}

\begin{remark}
    In \cite{gromov2007met}, the observable distance $\dconc(X,Y)$ between two mm-space $X$ and $Y$ is defined as the infimum of $\hausp{\kf^\lambda}(\lip1(X)\circ\phi,\lip1(Y)\circ\psi)$, where $\phi$ and $\psi$ run over all parameters of $\mu_X$ and $\mu_Y$, respectively. We see that $\dconc(X,Y) = \dconc(D(X),D(Y))$ for any $X,Y\in\X$.
\end{remark}

\begin{definition}[Concentrate, concentration topology]
    Let $X_n$, $n=1,2,\ldots$, and $X$ be geometric data sets. We say that the sequence $\{X_n\}_{n=1}^\infty$ \emph{concentrates} to $X$ if \begin{equation*}
        \lim_{n\to\infty} \dconc(X_n,X) = 0.
    \end{equation*}
    If $\{X_n\}_{n=1}^\infty$ concentrates to X, we write $X_n\xrightarrow{\conc}X$. We call the topology introduced by $\dconc$ on $\D$ the \emph{concentration topology}.
\end{definition}

\begin{remark}
    We could consider the different topology induced from a pseudometric defined as $\dconc(\X(X),\X(Y))$ for any $X,Y\in \D$. This topology is not finer than the concentration topology since \begin{equation*}
        \dconc(\X(\ast_{\{0\}}),\X(\ast_{\{1\}})) = \dconc(\ast,\ast) = 0 < 1 = \dconc(\ast_{\{0\}},\ast_{\{1\}}),
    \end{equation*} where $d_\ast\coloneqq 0$ and $\mu_\ast$ is the Dirac measure. Moreover, this topology does not coincide with the concentration topology. Indeed, there exists the following example. For $N=1,2,\ldots$, we defined N-points discrete geometric data set by \begin{equation*}
        X_N \coloneqq \left(\{1,\ldots,N\},\{d^N_m\mid m=1,2,\ldots,N\},\mu_N\right),
    \end{equation*} where $\mu_N$ is the normalized counting measure, and \begin{equation*}
        d^N_m(n) \coloneqq \begin{cases}
            0 & \mathrel{\mathrm{if}} n=m,                    \\
            1 & \mathrel{\mathrm{if}} n\neq m, 1\leq n\leq N.
        \end{cases}
    \end{equation*} Let us evaluate the observable distance between $X_N$ and $\ast_{\{1\}}$. We set the parameters $\phi_N\colon I\to X_n$ and $\psi\colon I\to\ast$ by\begin{equation*}
        \phi_N(t) \coloneqq \frac{\lfloor N\cdot t\rfloor}{N} + 1,\quad \psi(t) \coloneqq \ast
    \end{equation*} for $N=1,2,\ldots$, where $\lfloor x\rfloor$ is the maximal integer of less than or equal to $x$. Since \begin{align*}
        \kf^\lambda(d^N_m\circ\phi_N,1\circ\psi) & \leq \lambda\left(\{t\in I\mid d^N_m\circ\phi_N(t) \neq 1\circ\psi(t)\} \right) \\
                                                 & \leq \lambda\left(\left(\frac{m-1}{N},\frac{m}{N}\right]\right) = \frac{1}{N}
    \end{align*} for $m=1,2,\ldots$, we have \begin{equation*}
        \dconc(X_N,\ast_{\{1\}}) \leq \hausp{\kf^\lambda}(F_{X_N}\circ\phi_N,\{1\}\circ\psi) \leq \frac{1}{N} \to 0 \mathrel{\textrm{as}} N\to\infty.
    \end{equation*} However, $\{\X(X_N)\}_{N=1}^\infty$ does not concentrate to $\X(\ast_{\{1\}})$. For $N=1,2,\ldots$, we define $h_N\in\lip1(X_{2N})$ by\begin{equation*}
        h_N(n) \coloneqq \begin{cases}
            0 & \mathrel{\mathrm{if}} 1\leq n\leq N, \\
            1 & \mathrel{\mathrm{if}} N< n\leq 2N.
        \end{cases}
    \end{equation*} Since $\kf^{X_{2N}}(h_N,x) \geq 1/2$ for any $x\in \R$, we have \begin{equation*}
        \dconc(\X(X_{2N}),\X(\ast_{\{1\}})) \geq \frac{1}{2}.
    \end{equation*}
\end{remark}

\begin{remark}
    The metric space $(\D,\dconc)$ is not separable. In fact, for any subset $\eta\subset 2^\N$, the only existing parameter for $\ast_\eta$ is obtained as $\phi(t) \coloneqq \ast$. For any $\xi \subset 2^\N$ different from $\eta$, we have \begin{align*}
        \dconc(\ast_\eta,\ast_\xi) & = \hausp{\kf^\lambda}(\{\ast\mapsto n\mid n\in\eta\}\circ\phi,\{\ast\mapsto n\mid n\in\xi\}\circ\phi) \\
                                   & = \max\left\{\hausp{d_\R}(\eta,\xi),1\right\} = 1.
    \end{align*} The uncountability of the set $2^\N$ of all subsets of $\N$ implies that $(\D,\dconc)$ is not separable.
\end{remark}

It is known that the observable distance between mm-spaces has some good properties \cite{nakajima2022coupling}. Let us prove that the same properties of the observable distances between geometric data sets. We use several lemmas.

\begin{lemma}[{\cite[Lemma 4.2]{nakajima2022coupling}}]\label{lem:parameterEqualCoupling} For mm-spaces $X, Y$, we have
    \begin{equation*}
        \Pi(\mu_X,\mu_Y) = \left\{(\phi,\psi)_*\lambda \mid \phi \text{ is parameter of $\mu_X$}, \psi \text{ is parameter of $\mu_Y$}\right\}
    \end{equation*}
\end{lemma}

\begin{lemma}[{\cite[Lemma 5.3]{nakajima2022coupling}}]\label{lem:parameterToCouplingByFunction} For mm-spaces $X,Y$, parameters $\phi$ of $X$, $\psi$ of $Y$, and 1-Lipschitz functions $f\in\lip1(X),g\in\lip1(Y)$, we have \begin{equation*}
        \kf^{(\phi,\psi)_*\lambda}(f\circ\pr_1,g\circ\pr_2) = \kf^\lambda(f\circ\phi,g\circ\psi).
    \end{equation*}
\end{lemma}

By Lemma~\ref{lem:parameterToCouplingByFunction}, the following is clear.
\begin{corollary}[{\cite[Lemma 5.4]{nakajima2022coupling}}]For mm-spaces $X,Y$, and parameters $\phi$ of $X$, $\psi$ of $Y$, we have \begin{equation*}
        \hausp{\kf^{(\phi,\psi)_*\lambda}}(F_X\circ\pr_1,F_Y\circ\pr_2) = \hausp{\kf^\lambda}(F_X\circ\phi,F_Y\circ\psi).
    \end{equation*}
\end{corollary}

\begin{lemma}[{\cite[Lemma 5.6]{nakajima2022coupling}}]\label{lem:KFDistanceLessThanDp} Let $X,Y$ be metric spaces, $\mu,\nu$ Borel probability measures, and 1-Lipschitz functions $f,g\in\lip1(X,Y)$. Then we have \begin{equation*}
        |\kf^\mu(f,g) - \kf^\nu(f,g)| \leq 2\dpr(\mu,\nu).
    \end{equation*}
\end{lemma}

\begin{definition}
    We put\begin{equation*}
        \dconc^\pi(X,Y) \coloneqq \hausp{\kf^\pi}(F_X\circ\pr_1,F_Y\circ\pr_2)
    \end{equation*} for $\pi\in\Pi(\mu_X,\mu_Y)$.
\end{definition}

\begin{lemma}\label{lem:dconcIsContinuousByCoupling}
    For geometric data sets $X,Y$ and couplings $\pi,\rho\in\Pi(\mu_X,\mu_Y)$, \begin{equation*}
        |\dconc^\pi(X,Y)-\dconc^\rho(X,Y)| \leq 2\dpr(\pi,\rho),
    \end{equation*} where we equip $X\times Y$ with the $l^\infty$-product metric.
\end{lemma}
\begin{proof}
    By the symmetry of $\pi,\rho$, it is sufficient to prove that\begin{equation*}
        \dconc^\pi(X,Y)-\dconc^\rho(X,Y) \leq 2\dpr(\pi,\rho).
    \end{equation*} Take any $\epsilon>0$ and $f\in F_X$. There exists $g\in F_Y$ such that \begin{equation*}
        \kf^\rho(f\circ\pr_1,g\circ\pr_2) < \dconc^\rho(X,Y) + \epsilon.
    \end{equation*} Lemma~\ref{lem:KFDistanceLessThanDp} implies \begin{align*}
        \kf^\pi(f\circ\pr_1,g\circ\pr_2) & \leq \kf^\rho(f\circ\pr_1,g\circ\pr_2) + 2\dpr(\pi,\rho) \\
                                         & \leq \dconc^\rho(X,Y) + 2\dpr(\pi,\rho) + \epsilon.
    \end{align*} By the arbitrariness of $f$, we have $F_X\circ\pr_1\subset U_{\kf^\pi}(F_Y\circ\pr_2)$. From the symmetry of $f$ and $g$, we prove \begin{equation*}
        \dconc^\pi(X,Y) \leq \dconc^\rho(X,Y) + 2\dpr(\pi,\rho) + \epsilon.
    \end{equation*} The arbitrariness of $\epsilon$ implies $\dconc^\pi(X,Y)-\dconc^\rho(X,Y) \leq 2\dpr(\pi,\rho)$.
\end{proof}

\begin{theorem}\label{thm:dconcMinimum}
    For geometric data sets $X,Y$, we have\begin{equation*}
        \dconc(X,Y) = \min\{\hausp{\kf^\pi}(F_X\circ\pr_1,F_Y\circ\pr_2) \mid \pi\in\Pi(\mu_X,\mu_Y)\}.
    \end{equation*}
\end{theorem}
\begin{proof}
    Let $\phi$ be a parameter of $X$ and $\psi$ a parameter of $Y$. By Lemma~\ref{lem:parameterEqualCoupling}, we prove\begin{align*}
        \dconc(X,Y) & = \inf\{\hausp{\kf^\lambda}(F_X\circ\phi,F_Y\circ\psi) \mid \phi,\psi\}                   \\
                    & = \inf\{\hausp{\kf^{(\phi,\psi)_*\lambda}}(F_X\circ\pr_1,F_Y\circ\pr_2) \mid \phi,\psi\}.
    \end{align*} From Lemma~\ref{lem:parameterToCouplingByFunction}, this equals \begin{equation*}
        \inf\{\hausp{\kf^\pi}(F_X\circ\pr_1,F_Y\circ\pr_2) \mid \pi\in\Pi(\mu_X,\mu_Y)\}.
    \end{equation*} Lemma~\ref{lem:dconcIsContinuousByCoupling} implies the continuity of $\hausp{\kf^\pi}(F_X\circ\pr_1,F_Y\circ\pr_2)$ in $\pi\in\Pi(\mu_X,\mu_Y)$ with respect to $\dpr$. By Lemma~\ref{lem:CouplingCompact}, there exists $\pi\in\Pi(\mu_X,\mu_Y)$ such that \begin{equation*}
        \hausp{\kf^\pi}(F_X\circ\pr_1,F_Y\circ\pr_2) = \dconc(X,Y).
    \end{equation*}
\end{proof}

\subsection{The Semi-Continuity of Observable Diameter and the Conversation of Feature Order Relation with respect to the Observable Distance}
The semi-continuity of observable diameter with respect to the observable distance is obvious from the following proposition.

\begin{proposition}
    Let $X,Y$ be geometric data sets. For any real numbers $\delta > \dconc(X,Y)$ and $\kappa\in [0,1-\delta]$, we have\begin{equation*}
        \od\bigl(X;\ -(\kappa+\delta)\bigr) \leq \od(Y;\  -\kappa) + 2\delta.
    \end{equation*}
\end{proposition}
\begin{proof}
    For any $f\in F_X$ and $\epsilon > 0$, there exist $\pi\in\Pi(\mu_X,\mu_Y),g\in F_Y$ such that\begin{equation*}
        \kf^\pi(f\circ\pr_1,g\circ\pr_2) < \delta.
    \end{equation*} Lemma~\ref{lem:ProhorovLeqKyFan} implies\begin{align*}
        \dpr(f_*\mu_X,g*\mu_Y) & = \dpr(f_*(\pr_1)_*\pi,g_*(\pr_2)_*\pi) \\ &\leq \kf^\pi(f\circ\pr_1,g\circ\pr_2) < \delta. \\
    \end{align*}
    From this and Lemma~\ref{lem:partDiamParamEneq}, we have \begin{align*}
        \pd\bigl(f_*\mu_X;\ -(\kappa+\delta)\bigr) & \leq \pd(g_*\mu_Y;\  -\kappa) + 2\delta \\
                                                   & \leq \od(Y;\  -\kappa) + 2\delta.
    \end{align*} The arbitrariness of $f$ implies \begin{equation*}
        \od\bigl(X;\ -(\kappa+\delta)\bigr) \leq \od(Y;\  -\kappa) + 2\delta.
    \end{equation*}
\end{proof}

To show that the feature order relation is preserved in the concentration topology, we need some lemmas.
\begin{lemma}\label{lem:bddIsCompact}
    Let $X$ be an mm-space, $x_0\in X$, and $c\geq 0$ a real number. Then the set\begin{equation*}
        \bdd(X,x_0,c) \coloneqq \{f\in\lip1(X)\mid |f(x_0)| \leq c\}
    \end{equation*} is compact with respect to the pointwise topology and the topology of convergence in measure.
\end{lemma}
\begin{proof}
    The equicontinuity of $\lip1(X)$ is clear. For any $x\in X$ and $f\in\bdd(X,x_0,c)$, we have\begin{equation*}
        |f(x)| \leq d_X(x,x_0) + c.
    \end{equation*} This implies that $\bdd(X,x_0,c)$ is pointwise bounded. By the Arzelà--Ascoli theorem, $\bdd(X,x_0,c)$ is compact with respect to the pointwise topology. We have the compactness with respect to the topology of convergence in measure from Lemma~\ref{prop:pointwiseEqKf}.
\end{proof}

\begin{lemma}\label{lem:dconcFeatureFunction}
    For geometric data sets $X,Y$ and $\pi\in\Pi(\mu_X,\mu_Y)$, there exists a map $u\colon \overline{F_X}\to\overline{F_Y}$ such that\begin{equation*}
        \kf^\pi(f\circ\pr_1,u(f)\circ\pr_2) \leq \dconc(X,Y)
    \end{equation*} for any $f\in F_X$.
\end{lemma}
\begin{proof}
    Take any $f\in\overline{F_X}$ and let us find $g\in \overline{F_Y}$ such that\begin{equation*}
        \kf^\pi(f\circ\pr_1,g\circ\pr_2) \leq \dconc^\pi(X,Y).
    \end{equation*} In the case of $\dconc^\pi(X,Y) = 1$, it is clear that\begin{equation*}
        \kf^\pi(f\circ\pr_1,g\circ\pr_2) \leq 1 = \dconc^\pi(X,Y)
    \end{equation*} for any $g\in F_Y\neq \emptyset$. We assume $\dconc^\pi(X,Y) < 1$ and put $\epsilon \coloneqq (1-\dconc^\pi(X,Y))/3$.
    There exists a sequence $\{f_n\}_{n=1}^\infty\subset F_X$ such that $\kf^X(f_n,f) < \epsilon/n$ because $f\in\overline{F_X}$. By the definition of $\dconc^\pi(X,Y)$, there exists $\{g_n\}_{n=1}^\infty\subset F_Y$ such that \begin{equation*}
        \kf^\pi(f_n\circ\pr_1,g_n\circ\pr_2) \leq \dconc^\pi(X,Y) + \frac{\epsilon}{n}.
    \end{equation*} The triangle inequality of $\kf^\pi$ implies that\begin{equation*}
        \kf^\pi(f\circ\pr_1,g_n\circ\pr_2) \leq \dconc^\pi(X,Y) + \frac{2\epsilon}{n}.
    \end{equation*} We define a closed set as\begin{equation*}
        S_n\coloneqq \left\{(x,y)\in X\times Y\mmid |f(x)-g_n(y)| \leq \dconc^\pi(X,Y) + \frac{2\epsilon}{n}\right\}.
    \end{equation*} Then, we have $\pi(S_n) \geq 1-(\dconc^\pi(X,Y) + 2\epsilon/n) > 0$. In particular, $S_n$ is non-empty. By Theorem~\ref{thm:weakHausdorffSubSequence}, there exists a closed subset $S$ of $X\times Y$ and a subsequence $\{S_{\iota_1(n)}\}_{n=1}^\infty$ such that $S_{\iota_1(n)}$ converges to $S$ in the weak Hausdorff sense. Proposition~\ref{thm:weakHausdorffMeasureBound} implies $\pi(S) \geq 1-\dconc^\pi(X,Y) > 0$. In particular, $S$ is non-empty. For any $(x,y)\in S$, there exists $(x_n,y_n)\in S_{\iota_1(n)}$, $n=1,2,\ldots$, such that $(x_n,y_n)$ converges to $(x,y)$ as $n\to\infty$. Thus, we have\begin{align*}
        \left|f(x)-g_{\iota_1(n)}(y)\right| & \leq d_X(x,x_n) + \left|f(x_n)-g_{\iota_1(n)}(y_n)\right| + d_Y(y_n,y)       \\
                                            & \leq d_X(x,x_n) + \dconc^\pi(X,Y) + \frac{2\epsilon}{\iota_1(n)} +d_Y(y_n,y) \\
                                            & \to \dconc^\pi(X,Y)
    \end{align*} as $n\to\infty$. Take a point $(x_0,y_0)\in S\neq\emptyset$ and put \begin{equation*}
        L \coloneqq \sup \left\{ \left| f(x_0)-g_{\iota_1(n)}(y_0) \right| \mmid n=1,2,\ldots\right\} < +\infty.
    \end{equation*}
    By Lemma~\ref{lem:bddIsCompact} and $g_{\iota_1(n)}(y_0)\in\bdd(Y,y_0,|f_(x_0)|+L)$, there exists a $g\in F_Y$ and a subsequence $\left\{g_{\iota_1\circ\iota_2(n)}\right\}_{n=1}^\infty$ such that $g_{\iota_1\circ\iota_2(n)}$ converges to $g$ as $n\to\infty$. For all $(x,y)\in S$, \begin{align*}
        |f(x)-g(y)| & \leq \lim_{n\to\infty} |f(x) - g_{\iota_1\circ\iota_2(n)}(y)| \\
                    & \leq \dconc^\pi(X,Y).
    \end{align*} Since $\pi(S) \geq 1 - \dconc^\pi(X,Y)$, we have\begin{equation*}
        \kf^\pi(f\circ\pr_1,g\circ\pr_2) \leq \dconc^\pi(X,Y).
    \end{equation*}
\end{proof}

\begin{lemma}\label{lem:dconcOfDominatedSpace}
    Let $X,Y$ be geometric data sets, $Z$ a feature space, $\nu$ a Borel probability measure on $Z$, and $\phi\colon Z\to X$, $\psi\colon Z\to Y$ maps. If $\phi,\psi$ satisfy\begin{equation*}
        F_X\circ\phi\subset F_Z,\ F_Y\circ\psi\subset F_Z,\ \phi_*\nu = \mu_X,\ \psi_*\nu = \mu_Y,
    \end{equation*} then we have\begin{equation*}
        \dconc(X,Y) \leq \hausp{\kf^\nu}(F_X\circ\phi,F_Y\circ\psi).
    \end{equation*}
\end{lemma}
\begin{proof}
    We put $\pi\coloneqq (\phi,\psi)_*\nu$. Then $\pi$ is a coupling between $\mu_X$ and $\mu_Y$. For any $f\in F_X$, $g\in F_Y$, and $\epsilon > 0$, we have\begin{align*}
         & \pi(\left\{(x,y)\in X\times Y \mmid \left|f\circ\pr_1(x,y)-g\circ\pr_2(x,y)\right|\geq\epsilon\right\})                                          \\
         & = \nu\left((\phi,\psi)^{-1}\left(\bigl\{(x,y)\in X\times Y \mmid \left|f\circ\pr_1(x,y)-g\circ\pr_2(x,y)\right|\geq\epsilon\bigr\}\right)\right) \\
         & = \nu\left(\bigl\{z\in Z \mmid \left|f\circ\phi(z)-g\circ\psi(z)\right|\geq\epsilon\bigr\}\right).
    \end{align*} The arbitrariness of $\epsilon$ implies \begin{equation*}
        \kf^\pi(f\circ\pr_1,g\circ\pr_2)\leq\kf^\nu(f\circ\phi,g\circ\psi).
    \end{equation*} Since $f,g$ are arbitrary,\begin{equation*}
        \dconc(X,Y) \leq \hausp{\kf^\pi}(F_X\circ\pr_1,F_Y\circ\pr_2)\leq\hausp{\kf^\nu}(F_X\circ\phi,F_Y\circ\psi).
    \end{equation*}
\end{proof}

\begin{lemma}\label{lemma:dconc-hausdorff-func}
    Let $X,X',Y$ be geometric data sets. If $X'\preceq X$, then there exists a geometric data set $Y'$ such that\begin{equation*}
        \#F_{Y'} \leq \#F_{X'},\ Y'\preceq Y,\ \dconc(X',Y') \leq \dconc(X,Y).
    \end{equation*}
\end{lemma}
\begin{proof}
    Take a domination $\phi\colon X\to X'$. By Theorem~\ref{thm:dconcMinimum} and Lemma~\ref{lem:dconcFeatureFunction}, there exist $\pi\in\Pi(\mu_X,\mu_Y)$ and $u\colon \overline{F_X}\to \overline{F_Y}$ such that\begin{equation*}
        \kf^\pi(f\circ\pr_1,u(f)\circ\pr_2) \leq \dconc(X,Y)
    \end{equation*} for any $f\in\overline{F_X}$. We take $Y'\coloneqq Y/u(F_{X'}\circ\phi)$ and the quotient domination $\psi\colon Y\to Y'$. Since $F_{Y'}\circ\psi = u(F_{X'}\circ\phi)$ and by Lemma~\ref{lem:dconcOfDominatedSpace}, we have\begin{align*}
        \dconc(X',Y') & \leq \hausp{\kf^\pi}(F_{X'}\circ\phi\circ\pr_1,F_{Y'}\circ\psi\circ\pr_2)                   \\
                      & = \hausp{\kf^\pi}(F_{X'}\circ\phi\circ\pr_1,u(F_{X'}\circ\phi)\circ\pr_2) \leq \dconc(X,Y).
    \end{align*}
\end{proof}

\begin{theorem}\label{theorem:DconcPreservesDomination}
    Let $X,Y,X_n,Y_n$ be geometric data sets, $n = 1,2,\ldots$. If $X_n\preceq Y_n$ for $n=1,2,\ldots$, and if $X_n$ and $Y_n $ converge to $X$ and $Y$ as $n\to\infty$, respectively, then we have $X\preceq Y$.
\end{theorem}
\begin{proof}
    By Theorem~\ref{thm:dconcMinimum}, there exist couplings
    \begin{equation*}
        \pi_n\in\Pi(\mu_X,\mu_{X_n}) \text{ and } \rho_n\in\Pi(\mu_Y,\mu_{Y_n})
    \end{equation*}
    such that
    \begin{equation*}
        \hausp{\kf^{\pi_n}}(F_X\circ\pr_1,F_{X_n}\circ\pr_2) \leq \dconc(X,X_n)
    \end{equation*}
    and
    \begin{equation*}
        \hausp{\kf^{\rho_n}}(F_{Y_n}\circ\pr_1,F_Y\circ\pr_2) \leq \dconc(Y_n,Y).
    \end{equation*}
    Take a domination $\phi_n\colon Y_n\to X_n$. Since $\rho'_n\coloneqq (\phi_n\times\id_Y)_*\rho_n$ is a coupling of $\mu_{X_n}$ and $\mu_Y$, there exists a Borel probability measure $\eta'_n$ on $X\times X'\times Y$ such that \begin{equation*}
        (\pr_{1,2})_*\eta'_n = \pi_n,\ (\pr_{1,2})_*\eta'_n = \rho'_n.
    \end{equation*} By the gluing lemma. We put $\eta_n \coloneqq (\pr_{1,3})_*\eta'_n$, which is a coupling between $X$ and $Y$. By Lemma~\ref{lem:CouplingCompact}, there exists a subsequence $\{\eta_{\iota(n)}\}_{n=1}^\infty$ and a Borel probability measure $\eta$ on $Y$ such that $\eta_{\iota(n)}\to\eta$ as $n\to\infty$. For any $f\in F_X$ and $n=1,2,\ldots$, there exists $f_n\in \overline{F_{X_{\iota(n)}}}$ such that\begin{equation*}
        \hausp{\kf^{\pi_{\iota(n)}}}(f\circ\pr_1,f_n\circ\pr_2) \leq \dconc(X,X_n)
    \end{equation*} from Lemma~\ref{lem:dconcFeatureFunction}. By $f_n\circ\phi_{\iota(n)}\in\overline{F_Y}$ and Lemma~\ref{lem:dconcFeatureFunction}, there also exists $g_n\in\overline{F_{Y_{\iota(n)}}}$ such that\begin{equation*}
        \hausp{\kf^{\rho_{\iota(n)}}}(f_n\circ\phi_{\iota(n)}\circ\pr_1,g_n\circ\pr_2) \leq \dconc(Y_n,Y).
    \end{equation*} Lemma~\ref{lem:KFDistanceLessThanDp} implies \begin{align*}
         & \kf^\eta(f\circ\pr_1,g_n\circ\pr_2)                                                                                                      \\
         & \leq \kf^{\eta_{\iota(n)}}(f\circ\pr_1,g_n\circ\pr_2) + 2\dpr(\eta_{\iota(n)},\eta)                                                      \\
         & \leq \kf^{\pi_{\iota(n)}}(f\circ\pr_1,f_n\circ\pr_2) + \kf^{\rho'_{\iota(n)}}(f_n\circ\pr_1,g_n\circ\pr_2) + 2\dpr(\eta_{\iota(n)},\eta) \\
         & \leq \dconc(X,X_n) + \dconc(Y_n,Y) + 2\dpr(\eta_{\iota(n)},\eta)\to 0
    \end{align*} as $n\to\infty$. Since $(\pr_2)_*\eta=\mu_Y$, we see that $\{g_n\}_{n=1}^\infty$ is a Cauchy sequence with respect to $\kf^Y$. By the completeness of $(\lip1(Y),\kf^Y)$, there exists $g\in \overline{F_Y}$ such that $\kf^\eta(f\circ\pr_1,g\circ\pr_2) = 0$. We put \begin{equation*}
        G \coloneqq \left\{g\in\overline{F_Y}\mmid f\in F_X, \kf^\eta(f\circ\pr_1,g\circ\pr_2) = 0\right\},\ X'\coloneqq Y/G.
    \end{equation*} Then, we have $\hausp{\kf^\eta}(F_X\circ\pr_1,G\circ\pr_2) = 0$. From Lemma~\ref{lem:dconcOfDominatedSpace}, we obtain $\dconc(X,X') = 0$. This completes the pr This completes the proof.
\end{proof}

\section{Box Distance}
The box distance is a complete distance between two mm-spaces induced by Gromov. In this section, we also define the box distance between two geometric data sets and obtain its completeness.

Let $X$ be a set, $Y$ a metric space, and $A$ non-empty subset of $X$. We define a pseudometric $d_A$ on the total set of maps from $X$ to $Y$ as \begin{equation*}
    \dinf{A}(f,g) \coloneqq \sup\left\{d_Y(f(x),g(x))\mmid x\in A\right\}.
\end{equation*} Similarly, we define $\dinf{\emptyset} \equiv 0$.

\subsection{Definition}
\begin{definition}[Box distance]\label{def:BoxDistance}
    Let $X,Y$ be two geometric data sets. For a coupling $\pi\in\Pi(\mu_X,\mu_Y)$, feature subsets $F\subset \lip1(X)$ and $G\subset \lip1(Y)$, and a closed subset $S\subset X\times Y$, we define\begin{align*}
        \Box^S_\pi(F,G) & \coloneqq \max\left\{1-\pi(S),2\hauspdinf{S}(F\circ\pr_1,G\circ\pr_2)\right\} \\
        \Box_\pi(F,G)   & \coloneqq \inf\left\{\Box^S_\pi(F,G)\mmid S\subset \F(X\times Y)\right\}.
    \end{align*} We define the \emph{box distance} between $X$ and $Y$ as \begin{equation*}
        \Box(X,Y) \coloneqq \inf\left\{\Box_\pi(F_X,F_Y)\mmid \pi\in\Pi(\mu_X,\mu_Y)\right\}.
    \end{equation*}
\end{definition}

\begin{proposition}\label{prop:BoxDistanceLessThanDConc}
    For any geometric data sets $X,Y$, we have \begin{equation*}
        \dconc(X,Y) \leq \Box(X,Y).
    \end{equation*}
\end{proposition}
\begin{proof}
    Take any coupling $\pi\in\Pi(\mu_X,\mu_Y)$ and a closed subset $S\subset X\times Y$. For $f\in\lip1(X)$ and $g\in\lip1(Y)$, we have
    \begin{align*}
        \kf^\pi(f\circ\pr_1,g\circ\pr_2) & \leq \max\{1-\pi(S),\hauspdinf{S}(f\circ\pr_1,g\circ\pr_2)\}    \\
                                         & \leq \max\{1-\pi(S), 2\hauspdinf{S}(f\circ\pr_1,g\circ\pr_2)\}.
    \end{align*} The arbitrariness of $f$ and $g$ implies that\begin{align*}
        \hausp{\kf^\pi}(F_X\circ\pr_1,F_Y\circ\pr_2) & \leq \max\{1-\pi(S), 2\hauspdinf{S}(F_X\circ\pr_1,F_Y\circ\pr_2)\} \\
                                                     & = \Box^S_\pi(F_X,F_Y).
    \end{align*} Since $\pi$ and $S$ are arbitrary, we see that\begin{align*}
        \dconc(X,Y) & = \inf\left\{\kf^\pi(f\circ\pr_1,g\circ\pr_2)\mmid \pi\in\Pi(\mu_X,\mu_Y)\right\}                       \\
                    & \leq \inf\left\{\Box^S_\pi(F_X,F_Y)\mmid \pi\in\Pi(\mu_X,\mu_Y), S\in\F(X\times Y)\right\} = \Box(X,Y).
    \end{align*} This completes the proof.
\end{proof}

\begin{lemma}
    Let $X,Y$ be two geometric data sets. For any coupling $\pi\in\Pi(\mu_X,\mu_Y)$, subsets $F\subset \lip1(X)$ and $G\subset \lip1(Y)$, we have\begin{equation*}
        \Box_\pi(F,G) = \inf\left\{\Box^K_\pi(F,G)\mmid K\in\K(X\times Y)\right\}.
    \end{equation*}
\end{lemma}
\begin{proof}
    Take any closed subset $S\in\F(X\times Y)$ and real number $\epsilon > 0$. By the inner regularity of $\mu_X$, there exists a compact subset $K\in\K(S)$ such that $\mu_X(S\setminus K) < \epsilon$. For any maps $f,g\in\lip1(X\times Y)$, we have\begin{align*}
        \max\{1-\pi(S), 2\hauspdinf{S}(f,g)\}
         & \leq \max\{1-\pi(K), 2\hauspdinf{K}(f,g)\} + \epsilon
    \end{align*} The arbitrariness of $f,g$ implies \begin{align*}
        \Box^S_\pi(F,G) & \leq \Box^K_\pi(F,G) + \epsilon.
    \end{align*} Since $S$ and $\epsilon$ are arbitrary, we obtain \begin{equation*}
        \Box_\pi(F,G) \leq \inf\left\{\Box^K_\pi(F,G)\mmid K\in\K(X\times Y)\right\}.
    \end{equation*} The opposite inequality on the other side is obvious.
\end{proof}

\begin{lemma}
    Let $X$ be a metric space. For any compact subset $K$ of $X$ and subset $F\subset\lip1(X)$, we have $\hauspdinf{K}(F,\overline{F}) = 0$.
\end{lemma}
\begin{proof}
    Take any real number $\epsilon > 0$. There exists a finite subset $A$ of $K$ such that $K\subset B_X(A;\epsilon)$. For any map $f\in\overline{F}$, there exists a map $f'\in F$ such that $\dinf{A}(f,f') < \epsilon$. In particular, $\dinf{K}(f,f') < 2\epsilon$. By the arbitrariness of $f$, we have $\overline{F}\subset B_{\dinf{K}}(F;2\epsilon)$. The arbitrariness of $\epsilon$ implies $\hauspdinf{K}(F,\overline{F}) = 0$. This completes the proof.
\end{proof}

\begin{proposition}
    $(\D,\Box)$ is a metric space.
\end{proposition}
\begin{proof}
    Let us prove the well-definedness of $\Box$ as defined on $\D\times\D$. Take any geometric data sets $X_1,Y_1,Y_2$ with $Y_1\simeq Y_2$ and take an isomorphism $\phi\colon Y_1\to Y_2$. We obtain\begin{align*}
        \Box(X_1,Y_1) & = \inf\left\{\Box^K_\pi(F_{X_1},F_{Y_1})\mmid\pi\in\Pi(\mu_{X_1},\mu_{Y_1}),K\in\K(X_1\times Y_1)\right\}                                                      \\
                      & = \inf\left\{\Box^K_\pi(F_{X_1},\overline{F_{Y_1}})\mmid\pi\in\Pi(\mu_{X_1},\mu_{Y_1}),K\in\K(X_1\times Y_1)\right\}                                           \\
                      & = \inf\left\{\Box^K_{\pi}(F_{X_1},\overline{F_{Y_2}}\circ\phi)\mmid\pi\in\Pi(\mu_{X_1},\mu_{Y_1}),K\in\K(X_1\times Y_1)\right\}                                \\
                      & = \inf\left\{\Box^{(\id_X\times\phi)(K)}_{(\id_X\times\phi)_*\pi}(F_{X_1},\overline{F_{Y_2}})\mmid\pi\in\Pi(\mu_{X_1},\mu_{Y_1}),K\in\K(X_1\times Y_1)\right\} \\
                      & \geq \inf\left\{\Box^{L}_{\rho}(F_{X_1},\overline{F_{Y_2}})\mmid\rho\in\Pi(\mu_{X_1},\mu_{Y_1}),L\in\K(X_1\times Y_2)\right\} = \Box(X_1,Y_2).
    \end{align*} From the symmetry between $Y_1$ and $Y_2$ and the arbitrariness of $X_1$, we have\begin{equation*}
        \Box(X_1,Y_1) = \Box(X_1,Y_2) = \Box(Y_2,X_1) = \Box(Y_2,X_2) = \Box(X_2,Y_2)
    \end{equation*} even for any geometric data set $X_2$ with $X_1\simeq X_2$. We obtain the well-definedness of $\Box$.

    We prove that $\Box$ is a distance. The non-negativity and symmetry are clear from the definition, and the non-degeneracy follows from Proposition~\ref{prop:BoxDistanceLessThanDConc}. We prove the triangle inequality \begin{equation*}
        \Box(X,Z) \leq \Box(X,Y) + \Box(Y,Z)
    \end{equation*} for any geometric data sets $X,Y,Z$. For any real number $\epsilon > 0$, there exist \begin{align*}
        \pi_{X,Y}\in\Pi(\mu_X,\mu_Y) & , S_{X,Y}\in\F(X\times Y),           \\
        \pi_{Y,Z}\in\Pi(\mu_Y,\mu_Z) & \text{ and } S_{Y,Z}\in\F(Y\times Z)
    \end{align*} such that \begin{align*}
        \max\left\{1-\pi_{X,Y}(S_{X,Y}),2\hauspdinf{S}(F_X\circ\pr_1,F_Y\circ\pr_2)\right\} & \leq \Box(X,Y) + \epsilon, \\
        \max\left\{1-\pi_{Y,Z}(S_{Y,Z}),2\hauspdinf{S}(F_Y\circ\pr_1,F_Z\circ\pr_2)\right\} & \leq \Box(Y,Z) + \epsilon.
    \end{align*}
    By the gluing lemma, there exists a Borel probability measure $\rho$ on $X\times Y \times Z$ such that\begin{equation*}
        \left(\pr_{1,2}\right)_*\rho = \pi_{X,Y} \text{ and } \left(\pr_{2,3}\right)_*\rho = \pi_{Y,Z}.
    \end{equation*} Setting \begin{align*}
        \pi_{X,Z} \coloneqq \left(\pr_{1,3}\right)_*\rho,\ S_{X,Z} \coloneqq \overline{\left\{(x,z)\mmid(x,y)\in S_{X,Y},(y,z)\in S_{Y,Z}\right\}},
    \end{align*} we have\begin{align*}
        1-\pi_{X,Z}(S_{X,Z}) & = 1 - \rho((S_{X,Y}\times Z)\cap (X\times S_{Y,Z})) \\
                             & \leq 1-\pi_{X,Y}(S_{X,Y}) + 1-\pi_{Y,Z}(S_{Y,Z})    \\
                             & = \Box(X,Y) + \Box(Y,Z) + 2\epsilon.
    \end{align*} Let us evaluate the remaining term $\hauspdinf{S_{X,Z}}(F_X\circ\pr_1,F_Z\circ\pr_2)$. For $f\in F_X$, there exist $g\in F_Y$ and $h\in F_Z$ such that\begin{align*}
        d^{S_{X,Y}}(f\circ\pr_1,g\circ\pr_2) & \leq \hausp{d^{S_{X,Y}}}(F_X\circ\pr_1,F_Y\circ\pr_2) + \epsilon, \\
        d^{S_{Y,Z}}(g\circ\pr_1,h\circ\pr_2) & \leq \hausp{d^{S_{Y,Z}}}(F_Y\circ\pr_1,F_Z\circ\pr_2) + \epsilon.
    \end{align*} For any $(x,z)\in S_{X,Z}$, there exists $y\in Y$ such that $(x,y)\in S_{X,Y}$ and $(y,z)\in S_{Y,Z}$. Then,\begin{align*}
        |f(x)-h(z)| & \leq |f(x)-g(y)| + |g(y)-h(z)|                                                                                       \\
                    & \leq \hauspdinf{S_{X,Y}}(F_X\circ\pr_1,F_Y\circ\pr_2) + \hauspdinf{S_{Y,Z}}(F_Y\circ\pr_1,F_Z\circ\pr_2) + 2\epsilon \\
                    & \leq \frac{1}{2}\left(\Box(X,Y)+\Box(Y,Z)\right) + 2\epsilon.
    \end{align*} The arbitrariness of $x,z$ implies \begin{equation*}
        \dinf{S_{X,Z}}(f\circ\pr_1,h\circ\pr_2) \leq \frac{1}{2}\left(\Box(X,Y)+\Box(Y,Z)\right) + 2\epsilon.
    \end{equation*} From the arbitrariness of $f$ and the symmetry between $F_X$ and $F_Z$, we have\begin{equation*}
        \hausp{d^{S_{X,Z}}}(F_X\circ\pr_1,F_Z\circ\pr_2) \leq \frac{1}{2}\left(\Box(X,Y)+\Box(Y,Z)\right) + 2\epsilon.
    \end{equation*} Therefore, \begin{align*}
        \Box(X,Z) & \leq \max\left\{1-\pi_{X,Z}(S_{X,Z}),2\hauspdinf{S}(F_X\circ\pr_1,F_Z\circ\pr_2)\right\} \\
                  & \leq \Box(X,Y)+\Box(Y,Z) + 4\epsilon.
    \end{align*} By the arbitrariness of $\epsilon$, we obtain the triangle inequality. Then, $\Box$ is distance function on $\D$. This completes the proof.
\end{proof}
\begin{remark}
    $(\D,\Box)$ is not separable. Proposition~\ref{prop:BoxDistanceLessThanDConc} and non-separability of $\dconc$ prove it.
\end{remark}

\subsection{Characterization}
Let us see that Definition~\ref{def:BoxDistance} is an extension of the box distance between two mm-spaces, by using the characterization of the box distance in \cite[§4]{nakajima2022coupling}.

\begin{definition}[Distortion, {\cite[Definition 4.1]{nakajima2022coupling}}] Let $X,Y$ be two geometric data sets. For any $S\in\F(X\times Y)$, we define the \emph{distortion} of $S$ as
    \begin{equation*}
        \dis S \coloneqq \max\left\{\left|d_X(x_1,x_2)-d_Y(y_1,y_2)\right|\mmid (x_1,y_1),(x_2,y_2)\in S\right\}.
    \end{equation*} For any $\pi\in\Pi(\mu_X,\mu_Y)$, we define the \emph{distortion} of $\pi$ as\begin{align*}
        \dis\pi\coloneqq \inf\left\{\max\left\{1-\pi(S),\dis S\right\}\mmid S\in\F(X\times Y)\right\}.
    \end{align*}
\end{definition}

\begin{proposition}\label{prop:BoxDistanceLessThanDis}
    Let $X,Y$ be two geometric data sets and $\pi\in\Pi(\mu_X,\mu_Y)$. Then we have $\dis \pi \leq \Box_\pi(F_X,F_Y)$.
\end{proposition}
\begin{proof}
    For any real number $\epsilon > 0$, there exists $S\in\F(X\times Y)$ such that\begin{equation*}
        \Box_\pi(F_X,F_Y) \leq \max\left\{1-\pi(S),2\hauspdinf{S}(F_X\circ\pr_1,F_Y\circ\pr_2)\right\} + \epsilon.
    \end{equation*} Take any $(x_1,y_1),(x_2,y_2)\in S$. There exist $f\in F_X$ and $g\in F_Y$ such that \begin{align*}
        d_X(x_1,x_2)                      & \leq |f(x_1)-f(x_2)| + \epsilon,                            \\
        \dinf{S}(f\circ\pr_1,g\circ\pr_2) & \leq \hauspdinf{S}(F_X\circ\pr_1,F_Y\circ\pr_2) + \epsilon.
    \end{align*} From the triangle inequality, we have\begin{align*}
        d_X(x_1,x_2) & < |f(x_1)-f(x_2)| + \epsilon                                                 \\
                     & \leq |f(x_1)-g(y_1)| + |g(y_1)-g(y_2)| + |g(y_2)-f(x_2)| + \epsilon          \\
                     & \leq |g(y_1)-g(y_2)| + 2\dinf{S}(f\circ\pr_1,g\circ\pr_2) + \epsilon         \\
                     & \leq d_Y(y_1,y_2) + 2\hauspdinf{S}(F_X\circ\pr_1,F_Y\circ\pr_2) + 3\epsilon.
    \end{align*} The symmetry between $X$ and $Y$ proves\begin{equation*}
        |d_X(x_1,x_2) - d_Y(y_1,y_2)| \leq 2\hauspdinf{S}(F_X\circ\pr_1,F_Y\circ\pr_2) + 3\epsilon.
    \end{equation*} By the arbitrariness of $(x_1,y_1)$ and $(x_2,y_2)$, we obtain\begin{equation*}
        \dis S \leq 2\hauspdinf{S}(F_X\circ\pr_1,F_Y\circ\pr_2) + 3\epsilon.
    \end{equation*} The definition of $S$ proves\begin{align*}
        \dis\pi & \leq \max\left\{1-\pi(S),\dis S\right\}                                                  \\
                & \leq \max\left\{1-\pi(S),2\hauspdinf{S}(F_X\circ\pr_1,F_Y\circ\pr_2) + 3\epsilon\right\} \\
                & \leq \Box_\pi(F_X,F_Y) + 4\epsilon.
    \end{align*} Since $\epsilon$ is arbitrary, $\dis \pi \leq \Box_\pi(F_X,F_Y)$.  This completes the proof.
\end{proof}

\begin{proposition}
    Let $X,Y$ be two geometric data sets and $\pi\in\Pi(\mu_X,\mu_Y)$ a coupling. Then we have $\dis \pi = \Box_\pi(\lip1(X),\lip1(Y))$.
\end{proposition}
\begin{proof}
    By Proposition~\ref{prop:BoxDistanceLessThanDis}, we see $\dis \pi \leq \Box_\pi(\lip1(X),\lip1(Y))$. It is sufficient to prove that $\Box_\pi(\lip1(X),\lip1(Y)) \leq \dis \pi$.

    For any real number $\epsilon > 0$, there exists $S\in\F(X\times Y)$ such that\begin{equation*}
        \max\left\{1-\pi(S),\dis S\right\} \leq \dis\pi + \epsilon.
    \end{equation*} Let us evaluate $\hauspdinf{S}(\lip1(X)\circ\pr_1,\lip1(Y)\circ\pr_2)$. Take any $f\in\lip1(X)$ and define $g\colon Y\to\R$ by
    \begin{equation*}
        g(y) \coloneqq \frac{1}{2}\dis S + \inf_{(x,z)\in S} \left(f(x) + d_Y(z,y)\right),\quad y\in Y.
    \end{equation*}
    We prove $g\in\lip1(Y)$. For any $y_1,y_2\in Y$, \begin{align*}
        g(y_1) - g(y_2) & =  \inf_{(x,z)\in S} \left(f(x) + d_Y(z,y_1)\right) - \inf_{(x,z)\in S} \left(f(x) + d_Y(z,y_2)\right) \\
                        & \leq  \inf_{(x,z)\in S} \left(f(x) + d_Y(z,y_1) - f(x) - d_Y(z,y_2)\right)                             \\
                        & = \inf_{(x,z)\in S} \left(d_Y(z,y_1) - d_Y(z,y_2)\right) \leq d_Y(y_1,y_2).
    \end{align*} The arbitrariness of $y_1,y_2$ implies $g\in \lip1(Y)$. Let us see $\dinf{S}(f\circ\pr_1,g\circ\pr_2) \leq (\dis S) / 2$. For any $(x,y)\in S$, we have \begin{align*}
        g(y) - f(x) & \leq \frac{1}{2}\dis S + \left(f(x) + d_Y(y,y)\right) - f(x) = \frac{1}{2}\dis S                         \\
        f(x) - g(y) & = f(x) - \left(\frac{1}{2}\dis S + \inf_{(x',y')\in S} \left(f(x') + d_Y(y',y)\right)\right)             \\
                    & \leq \sup_{(x',y')\in S} \left(f(x) - f(x') - d_Y(y',y)\right) - \frac{1}{2}\dis S                       \\
                    & \leq \sup_{(x',y')\in S} \left(d_X(x,x') - d_Y(y,y')\right) - \frac{1}{2}\dis S  \leq \frac{1}{2}\dis S.
    \end{align*} By the arbitrariness of $x,y$, we have $\dinf{S}(f\circ\pr_1,g\circ\pr_2) \leq (\dis S)/2$. The symmetry between $X$ and $Y$ implies\begin{equation*}
        \hauspdinf{S}(\lip1(X)\circ\pr_1,\lip1(Y)\circ\pr_2) \leq \frac{1}{2}\dis S + \epsilon.
    \end{equation*}

    Finally, we have\begin{align*}
        \Box_\pi(\lip1(X),\lip1(Y)) & \leq \max\left\{1-\pi(S),2\hauspdinf{S}(\lip1(X)\circ\pr_1,\lip1(Y)\circ\pr_2)\right\} \\
                                    & \leq \max\left\{1-\pi(S),\dis S\right\} + 2\epsilon < \dis\pi + 3\epsilon.
    \end{align*} The arbitrariness of $\epsilon$ proves $\Box_\pi(\lip1(X),\lip1(Y)) \leq \dis\pi$. This completes the proof.
\end{proof}

In \cite[Proposition 4.4]{nakajima2022coupling}, the box distance is defined even between mm-spaces, and the following proposition is known:
\begin{proposition}[{\cite[Proposition 4.4]{nakajima2022coupling}}]\label{prop:Box_Use_Coupling} For mm-spaces $X,Y$, we have\begin{equation*}
        \Box(X,Y) = \inf\left\{\dis\pi\mmid\pi\in\Pi(\mu_X,\mu_Y)\right\}.
    \end{equation*}
\end{proposition}
\noindent In particular, we see\begin{equation*}
    \Box(X,Y) = \Box(\D(X),\D(Y))
\end{equation*} for any mm-spaces $X,Y$. For any geometric data sets $X,Y$, we have also \begin{equation*}
    \Box(\X(X),\X(Y)) \leq \Box(X,Y).
\end{equation*}

\begin{remark}
    The equality does not necessarily hold in the above inequality. In fact, for $n\in\N$, the introduced mm-space $\X(\ast_{\{n\}})$ of the singleton geometric data set $\ast_{\{n\}}$ equals $\ast$, where $\ast$ is the one-point mm-space $(\ast,0,\delta_\ast)$. For $m\in\N$ with $n\neq m$, \begin{align*}
        \Box(\X(\ast_{\{n\}}),\X(\ast_{\{m\}})) & = \Box(\ast,\ast) = 0 < 1 = \dconc(\ast_{\{n\}},\ast_{\{m\}}) \\ &\leq \Box(\ast,\ast) \leq \Box(\ast_{\{n\}},\ast_{\{m\}}).
    \end{align*}
\end{remark}

Proposition~\ref{prop:Box_Use_Coupling} means that the box distance between geometric data sets defined in this paper is a generalization of the box distance between mm-spaces due to Gromov. Next, let us characterize the box distance in this way as described in [6, §4].

\begin{lemma}\label{lem:BoxIsContinuousByCoupling}
    Let $X,Y$ be two geometric data sets and $\pi,\rho\in\Pi(\mu_X,\mu_Y)$ couplings. Then we have\begin{equation*}
        \left|\Box_\pi(X,Y) - \Box_\rho(X,Y)\right| \leq 4\dpr(\pi,\rho),
    \end{equation*} where $X\times Y$ is equipped with the $l^\infty$-product metric.
\end{lemma}
\begin{proof}
    For any real number $\epsilon > 0$, there exists $S_\pi\in\F(X\times Y)$ such that \begin{equation*}
        \max\left\{1-\pi(S_\pi),2\hauspdinf{S_\pi}(F_X\circ\pr_1,F_Y\circ\pr_2)\right\} < \Box_\pi(F_X,F_Y) + \epsilon.
    \end{equation*} Setting $S_\rho \coloneqq B_{X\times Y}(S;\dpr(\pi,\rho) + \epsilon)$, we see\begin{align*}
        1 - \rho(S_\rho) & \leq 1 - (\pi(S_\pi) - (\dpr(\pi,\rho) + \epsilon)) \\
                         & = 1 - \pi(S_\pi) + \dpr(\pi,\rho) + \epsilon.
    \end{align*} For any $f\in F_X,g\in F_Y$, we have\begin{align*}
        \dinf{S_\rho}(f\circ\pr_1,g\circ\pr_2) & \leq \dinf{S_\pi}(f\circ\pr_1,g\circ\pr_2) + 2(\dpr(\pi,\rho) + \epsilon).
    \end{align*} The arbitrariness of $f,g$ implies\begin{align*}
        \Box^S_\rho(F_X,F_Y) & \leq \max\left\{1-\pi(S_\rho),2\hauspdinf{S_\rho} (F_X\circ\pr_1,F_Y\circ\pr_2)\right\}                                          \\
                             & \leq \max\left\{1-\pi(S_\rho),2\left(\hauspdinf{S_\rho}(F_X\circ\pr_1,F_Y\circ\pr_2)\right)\right\}+4(\dpr(\pi,\rho) + \epsilon) \\
                             & < \Box_\pi(F_X,F_Y) + 4\dpr(\pi,\rho) + 5\epsilon.
    \end{align*} The symmetry between $\pi$ and $\rho$ and the arbitrariness of $\epsilon$ prove\begin{equation*}
        \left|\Box_\pi(X,Y) - \Box_\rho(X,Y)\right| \leq 4\dpr(\pi,\rho).
    \end{equation*} This completes the proof.
\end{proof}

\begin{proposition}\label{prop:BoxIsMinForCoupling}
    For geometric data sets $X,Y,$ we have \begin{equation*}
        \Box(X,Y) = \min\left\{\Box_\pi(F_X,F_Y)\mmid\pi\in\Pi(\mu_X,\mu_Y)\right\}.
    \end{equation*}
\end{proposition}
\begin{proof}
    The proposition follows from Lemmas~\ref{lem:CouplingCompact} and \ref{lem:BoxIsContinuousByCoupling}.
\end{proof}

\begin{lemma}\label{lem:MinBoxMap}
    Let $X$ be a metric space and $S_n$, $S$ closed sets on $X$, $n = 1,2,\ldots$~. We assume that $S_n$ converges to $S$ in the weak Hausdorff sense. Then, for any subsets $F,G\subset\lip1(X)$, there exists a map $u\colon \overline{F}\to\overline{G}$ such that\begin{equation*}
        \dinf{S}(f,u(f)) \leq \liminf_{n\to\infty} \hauspdinf{S_n}(F,G)
    \end{equation*} for any $f\in \overline{F}$.
\end{lemma}
\begin{proof}
    It is sufficient to show that there exists a $g\in\overline{G}$ such that\begin{equation*}
        \dinf{S}(f,g) \leq \liminf_{n\to\infty} \hauspdinf{S_n}(F,G)
    \end{equation*} for any $f\in\overline{F}$. If $S$ is empty, then $\dinf{S} \equiv 0$ proves this lemma for any $g\in\overline{G}$. Hence, we assume that $S$ is not empty. There exists a subsequence $\left\{S_{l(n)}\right\}_{n=1}^\infty$ such that\begin{equation*}
        \lim_{n\to\infty} \hauspdinf{S_{l(n)}}(F,G) = \liminf_{n\to\infty} \hauspdinf{S_n}(F,G) \eqqcolon \delta.
    \end{equation*} Take any $f\in\overline{F}$. For each $n=1,2,\ldots$, there exist maps $f_n\in F$ and $g_n\in G$ such that $f_n$ converges pointwise to $f$ as $n\to\infty$ and \begin{equation*}
        \dinf{S_{l(n)}}(f_n,g_n) < \hauspdinf{S_{l(n)}}(F,G) + \frac{1}{n}.
    \end{equation*} Take any $x\in S$. For each $n=1,2,\ldots$, there exists an $x_n\in S_n$ such that $x_n\to x$ as $n\to\infty$. Thus, we have \begin{align*}
        |g_n(x)-f(x)| & \leq
        |g_n(x) - g_n(x_{l(n)})| + |g_n(x_{l(n)}) - f_n(x_{l(n)})|                                  \\
                      & \quad   + |f_n(x_{l(n)})-f_n(x)| + |f_n(x)-f(x)|
        \\
                      & \leq 2d_X(x,x_{l(n)}) + \dinf{S_{l(n)}}(f_n,g_n) + |f_n(x)-f(x)| \to \delta
    \end{align*} as $n\to\infty$.

    Since $S$ is non-empty and $x$ is arbitrary, there exist an $s\in S$ and a real number $L>0$ such that $g_n\in\bdd(X,s,L)$ for any $n=1,2,\ldots$~. By Lemma~\ref{lem:bddIsCompact}, there exists a map $g\in \overline{G}$ and a subsequence $\left\{g_{l\circ m(n)}\right\}_{n=1}^\infty$ such that $g_{l\circ m(n)}$ converges pointwise to $g$ as $n\to\infty$. Since\begin{equation*}
        |f(x)-g(x)| = \lim_{n\to\infty} |f(x) - g_{l\circ m(n)}(x)| \leq \delta
    \end{equation*} and $x$ is arbitrary, we have $\dinf{S}(f,g) \leq \delta$. This completes the proof.
\end{proof}

\begin{theorem}\label{thm:BoxIsMin}
    For geometric data sets $X,Y$, we have\begin{equation*}
        \Box(X,Y) = \min\left\{\Box^S_\pi(\overline{F_X},\overline{F_Y})\mmid\pi\in\Pi(\mu_X,\mu_Y),S\in\F(X\times Y)\right\}.
    \end{equation*}
\end{theorem}
\begin{proof}
    From Proposition~\ref{prop:BoxIsMinForCoupling}, there exists a coupling $\pi\in\Pi(\mu_X,\mu_Y)$ such that $\Box(X,Y) = \Box_\pi(F_X,F_Y)$. For each $n=1,2,\ldots$, there exists $S_n\in\F(X\times Y)$ such that\begin{equation*}
        \max\left\{1-\pi(S_n), 2\hauspdinf{S_n}(F_X\circ\pr_1,F_Y\circ\pr_2)\right\} \leq \Box_\pi(F_X,F_Y) + \frac{1}{n}.
    \end{equation*} Since $X\times Y$ is separable and complete, from Theorem~\ref{thm:weakHausdorffSubSequence}, there exist a subsequence $\left\{S_{m(n)}\right\}_{n=1}^\infty$ and a closed set $S\in\F(X\times Y)$ such that $S_{m(n)}\to S$ in the weak Hausdorff sense as $n\to\infty$. By Proposition~\ref{thm:weakHausdorffMeasureBound} and Lemma~\ref{lem:MinBoxMap}, \begin{align*}
        \Box^S_\pi(\overline{F_X},\overline{F_Y})
         & = \max\left\{1-\pi(S), 2\hauspdinf{S}(\overline{F_X}\circ\pr_1,\overline{F_Y}\circ\pr_2)\right\}                                     \\
         & \leq \max\left\{1-\limsup_{n\to\infty} \pi(S_{m(n)}), 2\liminf_{n\to\infty}\hauspdinf{S_{m(n)}}(F_X\circ\pr_1,F_Y\circ\pr_2)\right\} \\
         & = \liminf_{n\to\infty} \Box^{S_{m(n)}}_\pi(F_X,F_Y) \leq \Box(X,Y).
    \end{align*}
    This completes the proof.
\end{proof}

\subsection{Completeness}

Let $\{A_i\}_{i\in I}$ be a family of sets and $I_0\subset I$ a subset. We define\begin{equation*}
    \pr^I_{I_0} \colon \prod_{i\in I} A_i \to \prod_{i\in I_0} A_i \text{ by } \pr^I_{I_0}(\{x_i\}_{i\in I}) \coloneqq \{x_i\}_{i\in I_0}.
\end{equation*}

\begin{lemma}[Kolmogorov consistency theorem, {\cite[Theorem 10.6.2]{cohn2013measure}}] \label{lem:kolmogorovConsistency} Let $\{X_i\}_{i\in I}$ be a family of separable complete spaces indexed by non-empty set $I$, $\mathcal{I}$ the total set of all non-empty finite subsets of $I$. Suppose that, for each $I_0\in\mathcal{I}$, there exists a measure $\mu_{I_0}$ satisfying the following conditions \textup{(1)} and \textup{(2)}.\begin{enumerate}
        \item $\mu_{I_0}$ is a Borel probability measure on $\prod_{i\in I_0} X_i$.
        \item For any $I_1,I_2\in\mathcal{I}$, $\mu_{I_2} = (\pr^{I_1}_{I_2})_*\mu_{I_1}$.
    \end{enumerate}
    Then, there exists a Borel probability measure
    \begin{equation*}
        \mu \text{ on } \prod_{i\in I} X_i \text{ such that } \mu_{I_0} = (\pr^I_{I_0})_*\mu \text{ for any } I_0\in\mathcal{I}.
    \end{equation*}
\end{lemma}

\begin{lemma}\label{lem:couplingChainExistence}
    Let $X_n$ be a geometric data set and $\pi_n\in\Pi(\mu_{X_n},\mu_{X_{n+1}})$ a coupling, $n=1,2,\ldots$~. There exists a Borel probability measure $\rho_m$ on $\prod_{n=1}^m X_n$ satisfying the following conditions \textup{(1)} and \textup{(2)}.\begin{enumerate}
        \item $(\pr_{\leq m})_*\rho_n = \rho_m$ for any natural numbers $m$ and $n$ with $1\leq m < n$.
        \item $(\pr_{n,n+1})_*\rho_n = \pi_n$ for any $n = 1,2,\ldots$~.
    \end{enumerate}
\end{lemma}
\begin{proof}
    Let us construct the measures $\rho_n$ inductively. For the first step, we put $\rho_1 \coloneqq \mu_{X_1}$ and $\rho_2 \coloneqq \pi_1$. It is trivial that these satisfy the conditions (1) and (2).

    For the inductive step, take $n\geq 2$ and assume that the measure $\rho_n$ satisfy the conditions (1) and (2). By the gluing lemma, there exists a Borel probability measure $\rho_{n+1}$ on $X_1\times\cdots\times X_{n+1}$ such that
    \begin{equation*}
        (\pr_{\leq n})_*\rho_{n+1} = \rho_n \text{ and } (\pr_{n,n+1})_*\rho_{n+1} = \pi_n.
    \end{equation*}
    Since $(\pr_{\leq m})_*\rho_{n+1} = (\pr_{\leq m})_*\rho_n = \rho_m$ for any natural number $m$ with $1 \leq m < n$, (1) holds. This completes the proof.
\end{proof}

\begin{lemma}
    For a sequence $\{X_n\}_{n=1}^\infty$ of geometric data sets and a sequence \begin{equation*}
        \{\pi_n\}_{n=1}^\infty \in \prod_{n=1}^\infty \Pi(\mu_{X_n},\mu_{X_{n+1}})
    \end{equation*} of couplings, there exists a Borel probability measure
    \begin{equation*}
        \mu \text{ on } \prod_{n=1}^\infty X_n \text{ such that } \pi_n = (\pr_{n,n+1})_*\mu \text{ for each } n=1,2,\ldots\text{.}
    \end{equation*}
\end{lemma}
\begin{proof}
    From Lemma~\ref{lem:couplingChainExistence}, there exists a sequence $\{\rho_n\}_{n=1}^\infty$ of Borel probability measures satisfying conditions (1) and (2). Setting $I\coloneqq\{1,2,3,\ldots\}$ and $\mathcal{I}$ to be the total set of finite subsets of $I$, we define \begin{equation*}
        \mu_{I_0} \coloneqq (\pr^{\{1,2,\ldots,\max I_0\}}_{I_0})_*\rho_{\max I_0}
    \end{equation*} for each $I_0\in\mathcal{I}$. Then, for any $I_1,I_2\in\mathcal{I}$ with $I_1\subset I_2$, we have\begin{align*}
        \mu_{I_2} & = (\pr^{\{1,2,\ldots,\max I_2\}}_{I_2})_*\rho_{\max I_2}                                                            \\
                  & = (\pr^{\{1,2,\ldots,\max I_2\}}_{I_2})_*(\pr^{\{1,2,\ldots,\max I_1\}}_{\{1,2,\ldots,\max I_2\}})_*\rho_{\max I_1} \\
                  & = (\pr^{\{1,2,\ldots,\max I_1\}}_{I_2})_*\rho_{\max I_1}                                                            \\
                  & = (\pr^{I_1}_{I_2})_*(\pr^{\{1,2,\ldots,\max I_1\}}_{I_1})_*\mu_{I_1} = (\pr^{I_1}_{I_2})_*\mu_{I_1}.
    \end{align*} By Lemma~\ref{lem:kolmogorovConsistency}, there exists a Borel probability measure $\mu$ on $\prod_{n=1}^\infty X_n$ such that $\mu_{I_0} = (\pr^I_{I_0})_*\mu$. In particular, we have\begin{align*}
        \pi_n & = (\pr_{n,n+1})_*\rho_{n+1} = (\pr_{n,n+1})_*\mu.
    \end{align*} This completes the proof.
\end{proof}

\begin{lemma}\label{lem:GreaterBox}
    Let $X_n$ be a geometric data set, $n=1,2,\ldots$, and assume \begin{equation*}
        \sum_{n=1}^\infty \Box(X_n,X_{n+1}) < +\infty.
    \end{equation*} Then, there exist a geometric data set $X_\infty$ and a domination $\phi_n\colon X_\infty \to X_n$ such that\begin{equation*}
        \hausp{\kf^{X_\infty}}(F_{X_n}\circ\phi_n,F_{X_{n+1}}\circ\phi_{n+1}) \leq \Box(X_n,X_{n+1})
    \end{equation*} for $n=1,2,\ldots$~.
\end{lemma}
\begin{proof}
    From Theorem~\ref{thm:BoxIsMin}, there exist a coupling $\pi_n\in\Pi(\mu_{X_n},\mu_{X_{n+1}})$ and a subset $S_n\in\F(X_n\times X_{n+1})$ such that\begin{equation*}
        \max\{1-\pi_n(S_n), 2\hauspdinf{S_n}(\overline{F_{X_n}}\circ\pr_1,\overline{F_{X_{n+1}}}\circ\pr_2)\} = \Box(X_n,X_{n+1}).
    \end{equation*} By Lemma~\ref{lem:couplingChainExistence}, there exists a Borel probability measure $\mu$ on $\prod_{n=1}^\infty X_n$ such that $\pi_n = (\pr_{n,n+1})_*\mu$ for any $n=1,2,\ldots$~. We set \begin{equation*}
        X'_\infty \coloneqq \bigcup_{N=1}^\infty \left\{\{x_n\}_{n=1}^\infty\mmid \{(x_n,x_{n+1})\}_{n=1}^\infty\in\prod_{n=1}^{N-1} \supp\pi_n\times\prod_{n=N}^\infty S_n\right\}.
    \end{equation*} Let us construct the distance function on $X'_\infty$. We define\begin{equation*}
        d_{X'_\infty}\left(\{x_n\}_{n=1}^\infty,\{y_n\}_{n=1}^\infty\right) \coloneqq \sup_{n=1}^\infty\ d_{X_n}(x_n,y_n)\in[0,+\infty]
    \end{equation*} for each $\{x_n\}_{n=1}^\infty,\{y_n\}_{n=1}^\infty\in X'_\infty$. Since there exists a natural number $N$ such that\begin{equation*}
        (x_n,x_{n+1}),(y_n,y_{n+1}) \in S_n
    \end{equation*} for all $n=N,N+1,\ldots$, we have\begin{equation*}
        \left|d_{X_n}(x_n,y_n)-d_{X_{n+1}}(x_{n+1},y_{n+1})\right| \leq \dis S_n \leq \Box(X_n,X_{n+1}).
    \end{equation*} Thus, the sequence $\{d_{X_n}(x_n,y_n)\}_{n=1}^\infty$ is bounded and $d_{X'_\infty}$ defines a metrization of the relative topology of the product topology.

    Let us construct a separable complete metric space $X_\infty$. We define similarly a distance function $d'$ on $\prod_{n=1}^\infty X_n$ by\begin{equation*}
        d'\left(\{x_n\}_{n=1}^\infty,\{y_n\}_{n=1}^\infty\right) \coloneqq \sup_{n=1}^\infty \frac{d_{X_n}(x_n,y_n)}{n\cdot(1+d_{X_n}(x_n,y_n))}
    \end{equation*} for each $\{x_n\}_{n=1}^\infty,\{y_n\}_{n=1}^\infty\in\prod_{n=1}^\infty X_n$. Then, $d'$ is a separable, complete metrization of the product topology. The completion of $X'_\infty$ can be topologically embedded in the closure $\overline{X'_\infty}$ with respect to the product topology because $d_{X'_\infty}$ is greater than or equals to $d'$ on $X'_\infty$. We denote by $X_\infty\subset \overline{X'_\infty}$ the embedded completion of $X'_\infty$. The separability of $d'$ proves also the separability of $X_\infty$.

    Define a structure of geometric data set on $X_\infty$. The restriction of $\mu|_{X_\infty}$ is a Borel probability measure on $X_\infty$ because we have \begin{align*}
        \mu(X'_\infty) & \geq \sup_{N=1}^\infty \mu\left(\left\{\{x_n\}_{n=1}^\infty\mmid \{(x_n,x_{n+1})\}_{n=1}^\infty\in\prod_{n=1}^{N-1} \supp\pi_n\times\prod_{n=N}^\infty S_n\right\}\right) \\
                       & \geq \sup_{N=1}^\infty \mu\left(\bigcap_{n=N}^\infty X_1\times\cdots\times X_{n-1}\times S_n \times X_{n+2}\times\cdots\right)                                            \\
                       & \geq \sup_{N=1}^\infty \left(1-\sum_{n=N}^\infty\Box(X_n,X_{n+1})\right) = 1.
    \end{align*} We set $F_{X_\infty} \coloneqq \lip1(X_\infty)$, $\mu_{X_\infty} \coloneqq \mu|_{X_\infty}$, and $\phi_n \coloneqq \pr_n|_{X_\infty}$ for each $n=1,2,\ldots$~.

    By the definition of $d_{X'_\infty}$, we see that $\phi_n$ is 1-Lipschitz continuous. Take any $n=1,2,\ldots$. Since \begin{align*}
        (\phi_n)_*\mu_{X_\infty} & = (\pr_n)_*\mu = \mu_{X_n},
    \end{align*} we see that $\phi_n$ is domination.

    Consider the distace between $F_{X_n}\circ\phi_n$ and $F_{X_{n+1}}\circ\phi_{n+1}$. From the definition of $\mu_{X_\infty}$, we have
    \begin{align*}
        \hausp{\kf^{X_\infty}}(F_{X_n}\circ\phi_n,F_{X_{n+1}}\circ\phi_{n+1})
         & = \hausp{\kf^{\pi_n}}(F_{X_n}\circ\pr_1,F_{X_{n+1}}\circ\pr_2)                        \\
         & = \hausp{\kf^{\pi_n}}(\overline{F_{X_n}}\circ\pr_1,\overline{F_{X_{n+1}}}\circ\pr_2).
    \end{align*}
    By the definition of $\kf^{\pi_n}$, we obtain that
    \begin{align*}
             & \hausp{\kf^{\pi_n}}(\overline{F_{X_n}}\circ\pr_1,\overline{F_{X_{n+1}}}\circ\pr_2)                                        \\
        \leq & \max\{1-\pi_n(S_n),\hauspdinf{S_n}(\overline{F_{X_n}}\circ\pr_1,\overline{F_{X_{n+1}}}\circ\pr_2)\}                       \\
        \leq & \max\{1-\pi_n(S_n),2\hauspdinf{S_n}(\overline{F_{X_n}}\circ\pr_1,\overline{F_{X_{n+1}}}\circ\pr_2)\} = \Box(X_n,X_{n+1}).
    \end{align*}
    Thus,
    \begin{equation*}
        \hausp{\kf^{X_\infty}}(F_{X_n}\circ\phi_n,F_{X_{n+1}}\circ\phi_{n+1}) \leq \Box(X_n,X_{n+1}).
    \end{equation*}
    This completes the proof.
\end{proof}

\begin{lemma}\label{lem:dInftyUniformlyContinuous}
    Let $X$ be a geometric data set, $K$ a compact subset of $X$, and $Y$ a metric space. Then, $d^K_\infty$ is uniformly continuous with respect to $(\lip1(X,Y),\kf^X)$.
\end{lemma}
\begin{proof}
    Take any real number $\epsilon > 0$. There exists finitely many points
    \begin{equation*}
        x_1,x_2,\ldots,x_N\in X \text{ such that } K\subset\bigcup_{n=1}^N U_X(x_n;\epsilon).
    \end{equation*}
    We set \begin{equation*}
        s \coloneqq \min\left\{\mu_X(U_X(x_n;\epsilon))\mid n=1,\ldots,N\right\},\ \delta \coloneqq \min\left\{\frac{s}{2},\epsilon\right\}.
    \end{equation*} We take any $f,g\in\lip1(X,Y)$ and $x\in K$ and assume that $\kf^X(f,g) < \delta$. There exists $n\in\{1,2,\ldots,N\}$ such that $d_X(x,x_n) < \epsilon$. Since
    \begin{equation*}
        \mu_X(\left\{x\in X\mmid d_Y(f(x),g(x)) < \delta\right\}) \geq 1-\delta
    \end{equation*}
    and
    \begin{equation*}
        \mu_X(U_X(x_n;\epsilon)) \geq s \geq 2\delta,
    \end{equation*}
    there exists $x'\in U_X(x_n;\epsilon)$ such that $d_Y(f(x'),g(x')) < \delta \leq \epsilon$. Thus, we have\begin{equation*}
        d_Y(f(x),g(x)) \leq d_Y(f(x'),g(x')) + 2d_X(x,x') < 3\epsilon.
    \end{equation*} The arbitrariness of $x$ proves that $d^K_\infty(f,g) < 3\epsilon$. We take any $f',g'\in\lip1(X,Y)$ that $\kf^X(f,g) < \delta$. We have $d^K_\infty(f,g) < 3\epsilon$ in the same way as above. From the triangle inequality, we have\begin{align*}
        d^K_\infty(f,f') & \leq d^K_\infty(f,g) + d^K_\infty(g,g') + d^K_\infty(g',f') \\
                         & \leq d^K_\infty(g,g') + 6\epsilon.
    \end{align*} The symmetry between $(f,f')$ and $(g,g')$ implies $d^K_\infty(g,g') \leq \dinf{S}(f,f') + 6\epsilon$. Thus, we have \begin{equation*}
        |d^K_\infty(f,f') - d^K_\infty(g,g')| \leq 6\epsilon.
    \end{equation*} The arbitrariness of $f,f',g,g'$ completes the proof.
\end{proof}

\begin{lemma}\label{lem:embeddedBox}
    Let $X,Y$ be two geometric data sets, $Z$ a feature space, $\nu$ a Borel probability measure on $Z$, and $\phi\colon Z\to X$, $\psi\colon Z\to Y$ two maps. We assume that\begin{equation*}
        F_X\circ\phi\subset F_Z,\ F_Y\circ\psi\subset F_Z,\ \phi_*\nu = \mu_X,\ \psi_*\nu = \mu_Y.
    \end{equation*} Then, we have\begin{equation*}
        \Box(X,Y) \leq \max\left\{ 1-\nu(S), 2\dinf{S}(F_X\circ\phi,F_Y\circ\psi)\right\}.
    \end{equation*}
\end{lemma}
\begin{proof}
    Since the measure $(\phi,\psi)_*\nu$ is a coupling between $\mu_X$ and $\mu_Y$, we have\begin{align*}
        \Box(X,Y) & \leq \max\left\{ 1-(\phi,\psi)_*\nu((\phi,\psi)(S)), 2\dinf{\overline{(\phi,\psi)(S)}}(F_X\circ\pr_1,F_Y\circ\pr_2)\right\} \\
                  & \leq \max\left\{ 1-\nu(S), 2\dinf{S}(F_X\circ\phi,F_Y\circ\psi)\right\}.
    \end{align*} This completes the proof.
\end{proof}

\begin{theorem}
    $(\D,\Box)$ is complete.
\end{theorem}
\begin{proof}
    Take any Cauchy sequence $\{X_n\}_{n=1}^\infty$ of geometric data sets. There exists a subsequence $\{X_{m(n)}\}_{n=1}^\infty$ such that $\Box(X_{m(n)},x_{m(n+1)}) < 2^{-n}$. By Lemma~\ref{lem:GreaterBox}, there exists a geometric data set $Y$ and a domination $\phi_n\colon Y \to X_n$, $n=1,2,\ldots$, such that\begin{equation*}
        \hausp{\kf^Y}(F_{X_{m(n)}}\circ\phi_n,F_{X_{m(n+1)}}\circ\phi_{n+1}) < 2^{-n}.
    \end{equation*} Since $\{F_{X_{m(n)}}\circ\phi_n\}_{n=1}^\infty$ is a Cauchy sequence in $(\F(\lip1(Y)),\hausp{\kf^Y})$, there exists a closed subset $F_\infty\in\F(\lip1(Y))$ such that $F_{X_{m(n)}}\circ\phi_n$ converges to $F_\infty$ as $n\to\infty$. We set $X\coloneqq Y/F_\infty$ and $\phi$ to be the quotient domination of it.

    Let us prove that $X_n$ converges $X$ as $n\to\infty$. Take any real number $\epsilon > 0$. From the inner regularity of $\mu_Y$, there exists a compact set $K\subset Y$ such that $\mu_Y(K) > 1-\epsilon$. By Lemma~\ref{lem:dInftyUniformlyContinuous}, for all sufficiently large natural number $n$,\begin{equation*}
        \hauspdinf{K}(F_{X_n}\circ\phi_n,F_{X_{n+1}}\circ\phi_{n+1}) < \frac{\epsilon}{2}.
    \end{equation*} From Lemma~\ref{lem:embeddedBox}, we have\begin{equation*}
        \Box(X_n,X) \leq \max\left\{1-\mu_Y(K),\ 2\hauspdinf{K}(F_{X_n}\circ\phi_n,F_X\circ\phi)\right\} < \epsilon.
    \end{equation*} This completes the proof.
\end{proof}

\bibliographystyle{abbrv}
\bibliography{all}

\end{document}